\documentclass[11pt,a4]{article}
\usepackage{amsmath,amsthm,amsfonts,amssymb,amscd}
\usepackage{bbm}
\usepackage{enumerate}
\usepackage{enumitem}
\usepackage{url}
\usepackage{cite}
\usepackage{fullpage}
\usepackage{fancyhdr}
\usepackage{amsmath,amsthm,amsfonts,amssymb,amscd}
\usepackage{color}
\usepackage{float}
\usepackage{soul}
\usepackage{graphicx}
\usepackage[margin=1in]{geometry}

\theoremstyle{plain}
\newtheorem{theorem}{Theorem}[section]
\newtheorem{lemma}[theorem]{Lemma}
\newtheorem{corollary}[theorem]{Corollary}
\newtheorem{proposition}[theorem]{Proposition}
\newtheorem{definition}{Definition}[section]
\newtheorem{remark}{Remark}[section]
\newtheorem{example}{Example}[section]

\newtheorem{Line}{Linesearch}
\newtheorem{method}{Method}
\newtheorem{fact}[theorem]{Fact}
\usepackage{cite}
\theoremstyle{definition}

\def\disp{\displaystyle}

\newcommand{\la}{\langle}
\newcommand{\ra}{\rangle}
\newcommand{\ve}{\varepsilon}

\newcommand{\dom}{{\rm dom\,}}

\newcommand{\al}{\alpha}
\newcommand{\prox}{{\rm prox}}
\newcommand{\nexto}{\kern -0.54em}

\newcommand{\dZ}{{\cal Z \kern -0.7em Z}}
\newcommand{\dC}{{\rm\hbox{C \kern-0.8em\raise0.2ex\hbox{\vrule height5.4pt width0.7pt}}}}
\newcommand{\dQ}{{\rm\hbox{Q \kern-0.85em\raise0.25ex\hbox{\vrule height5.4pt width0.7pt}}}}

\newcommand{\Id}{\ensuremath{\operatorname{Id}}}
\usepackage{epsfig}
\usepackage{latexsym}

\newenvironment{retraitsimple}{\begin{list}{--~}{
 \topsep=0.3ex \itemsep=0.3ex \labelsep=0em \parsep=0em
 \listparindent=1em \itemindent=0em
 \settowidth{\labelwidth}{--~} \leftmargin=\labelwidth
}}{\end{list}}

\newcommand{\lqqd}{\hfill{$\Box$}\bigskip}

\newcommand{\NN}{\mathbb{N}}
\newcommand{\HH}{\mathcal{H}}
\newcommand{\RR}{\mathbbm{R}}
\newcommand{\PP}{\mathcal{P}}
\newcommand{\QQ}{\mathcal{Q}}

\newcommand{\dsty}{\displaystyle}

\begin{document}
\title{On the convergence of the {forward-backward} splitting method with linesearches
	%\footnote{This manuscript has not been published or submitted simultaneously for publication elsewhere.}
	}
\author{Jos\'e Yunier Bello Cruz\footnote{Department of Mathematical Sciences, Northern Illinois University, WH 366, DeKalb, IL - 60115, USA (E-mail: {\tt yunierbello@niu.edu}).}\and Tran T.A.
Nghia\footnote{Department of Mathematics and Statistics, Oakland
University, Rochester, MI. 48309, USA. E-mail: nttran@oakland.edu}}
\maketitle 
\vspace{-5.0mm}

\begin{abstract}
\noindent In this paper we focus on the convergence analysis of the
forward-backward splitting method for solving nonsmooth
optimization problems in Hilbert spaces when the objective function
is the sum of two convex functions. Assuming that one of the
functions is Fr\'echet differentiable and using two new
linesearches, the weak convergence is established without any
Lipschitz continuity assumption on the gradient. Furthermore, we obtain many complexity results of cost values at the
iterates when the stepsizes are bounded below by a positive
constant.
\end{abstract}

\medskip


\medskip

\section{ Introduction}
We are interested in solving problems of the following form:
\begin{equation}\label{prob}
\min \,f(x)+g(x) \;\mbox{ subject to } \;x\in \HH,
\end{equation} where $\HH$ is a real Hilbert space with the inner product $\la \cdot, \cdot\ra$, and $f,g:\HH\rightarrow
\,\overline{\RR}:=\RR\cup\{+\infty\}$ are two proper lower
semicontinuous convex functions in which $f$ is Fr\'echet
differentiable on  an open set containing the domain of $g$. The optimal solution set of
this problem will be denoted by $S_*$. Recently problem \eqref{prob}
together with many variants of it has received much attention
from optimization community due to its broad applications to many
disciplines such as optimal control, signal processing, system
identification, machine learning, and image analysis; see,
\emph{e.g.}, \cite{neal-boyd, comb-2005, comb-2011} and the references
therein. Many effective methods have been proposed to solve problem
\eqref{prob}. Most of them keep using the idea of splitting $f$ and
$g$ separately and taking the advantage of some Lipschitz assumption
on the derivative  of $f$ at each iteration. Here we focus our
attention on the so-called \emph{forward-backward splitting
method}, which contains a forward gradient step of $f$ (an explicit
step) followed by a backward proximal step of $g$ (an implicit step)
for problem \eqref{prob}; see, \emph{e.g.}, \cite{neal-boyd}. In this work linesearches are used to
eliminate the undesired Lipschitz assumption on the gradient of $f$ mostly imposed in the
literature.

To describe and motivate our methods, let us recall here  the
so-called proximal operator $\prox_g:=(\partial g+\Id)^{-1}$, where
$\partial g$ is the classical convex subdifferential of $g$ and
$\Id$ is the identity operator in  $\HH$. Among many important
properties of proximal operators, it is well-known that $\prox_g$ is
well-defined with full domain, single-valued, and even nonexpansive;
see, \emph{e.g.}, \cite{comb-2005,comb-2011, librobauch}.
Furthermore, for any $\alpha>0$, $x$ is an optimal solution to
problem \eqref{prob} if and only if $x=\prox_{\alpha
g}(x-\alpha\nabla f(x))$. This indeed motivates the construction of
the iterative sequence forming the forward-backward
iteration as following:
\begin{equation}\label{F-B-I}
x^{k+1}:=\prox_{\alpha_k g}(x^k-\alpha_k\nabla f(x^k))
\end{equation}
with  positive stepsize $\alpha_k$. The iteration presented in
\eqref{F-B-I} has been attracted extensive interests due to its
simplicity and  several important advantages. It is well-known that
this method uses little storage, readily exploits the separable
structure of problem \eqref{prob}, and is easily implemented to
practical applications; see \cite{neal-boyd, nesterov-2013,beck-teu}. Moreover, scheme
\eqref{F-B-I} may reduce to many popular optimization methods as
particular cases including the projected gradient method for smooth
constrained minimization; the proximal point method;  the CQ
algorithm for the split feasibility problem; the projected Landweber
algorithm for constrained least squares; the iterative soft
thresholding algorithm for linear inverse problems; decomposition
methods for solving variational inequalities; and the simultaneous
orthogonal projection algorithm for the convex feasibility problem;
see, \emph{e.g.}, \cite{eicke, comb-1994, dau, beck, Rock1976,
tseng1991, tseng1990} and the references therein.

The convergence of the iteration \eqref{F-B-I} to an optimal
solution of \eqref{prob} is usually established under the assumption
that the gradient of $f$ is Lipschitz continuous and the stepsize
$\alpha_k$ is taken bounded below and less than some constant related with the
Lipschitz modulus; see, \emph{e.g.},
\cite[Theorem~3.4(i)]{comb-2005}. In this case, the main machinery
to prove the convergence and its complexity  is based on the
renowned \emph{Baillon-Haddad Theorem}
\cite[Corollary~18.16]{librobauch}. When $\nabla f$ is  Lipschitz continuous but somehow the Lipschitz constant is not
known, finding the stepsize $\alpha_k$ that guarantees the
convergence of \eqref{F-B-I} would be a challenge. However, the
following linesearch proposed in \cite{beck-teu} overcome this
inconvenience: choosing the stepsize $\alpha_k$ in \eqref{F-B-I} as
the largest $\alpha \in
\{\sigma,\sigma\theta,\sigma\theta^2,\ldots,\}$ with constants
$\sigma>0$ and $ \theta\in(0,1)$ such that:
\begin{equation}\label{conda1}
f(J(x^k,\alpha))\le f(x^k)+\la\nabla
f(x^k),J(x^k,\alpha)-x^k\ra+\frac{1}{2\alpha}\|x^k-J(x^k,\alpha)\|^2,
\end{equation} where $J(x^k,\alpha):=\prox_{\alpha g}(x^k-\alpha\nabla f(x^k))$ and $\|\cdot\|$ is the norm induced by the inner product in
$\HH$.
This linesearch is well-defined by taking the
advantage of the Lipschitz assumption for $\nabla f$ again via the
so-called {\em Descent Lemma} \cite[Theorem~18.15(iii)]{librobauch}.
As far as we observe, the theory of convergence and complexity for
the forward-backward is almost complete under such a
Lipschitz assumption. However, the Lipschitz condition fails in many natural circumstances;
see, \emph{e.g.}, \cite{D-R-comb}. It is quite interesting to
question the convergence of the method and its complexity without
the Lipschitz assumption aforementioned. In \cite{tseng} Tseng
provided an evidence of positive answer even for more general
problems of finding a zero point of the sum of two  maximal monotone operators. His crucial approach motivates us to construct
{\bf Method~\ref{A1}} for problem \eqref{prob} in our Section 4. But working on the  functionals ($f$ and $g$) rather than just the maximal operators ($\nabla f$ and $\partial g$)
actually gives us much more convenience. Indeed, we completely relax an
(expensive) extra projection step from Tseng's scheme and omit several
unnatural assumptions in the main theorem \cite[Theorem~3.4]{tseng}.
Moreover, in the spirit of linesearch on functionals like
\eqref{conda1} and following some ideas presented in \cite{yun-wlo, tseng2001, tseng-yun}, we also introduce a new linesearch mainly used in
our {\bf Method~\ref{A2}} in Section 5. Both  {\bf Method~\ref{A1}} and {\bf
Method~\ref{A2}} guarantee weak convergence of their generated sequences to
optimal solutions without imposing the Lipschitz assumption on $\nabla
f$.

Another achievement of our work is the study on complexity of cost
values at generated sequences, which are proved to converge to the infimum value of problem \eqref{prob} even in the case when the set of optimal solutions is empty.  It is worth mentioning that in order to obtain the rate
$\mathcal{O}(k^{-1})$ of the functional value $(f+g)(x^k)$ to the optimal cost, the gradient $\nabla f$
is usually supposed to be globally Lipschitz continuous in the classical forward-backward iteration
\cite{neal-boyd, nesterov-2013, beck, beck-teu, comb-2005}. Here, in finite dimensions, we
derive the better rate $o(k^{-1})$ even with strictly weaker assumptions, for
instance,  $\nabla f$ only needs to be  {\em locally} Lipschitz continuous for our {\bf Method~\ref{A1}} and {\bf Method~\ref{A2}}. {This partially generalizes several results in \cite{Dong2, Dong3, guller}, in which the authors also derive the complexity $o(k^{-1})$ for proximal point method (when $f\equiv 0$).}  Moreover, we present an interesting example of problem \eqref{prob} with non-Lipschitz gradient where the stepsizes generated by both linesearches converge to zero and the complexity $o(k^{-1})$ of the cost values remains valid. Furthermore, the
rate $\mathcal{O}\left(k^{-2}\right)$ is also obtained for  our {\bf Method~\ref{A3}}, an accelerating version of  {\bf Method~\ref{A1}} motivated from \cite{beck-teu}. Again, global
Lipschitz continuity on $\nabla f$ is lessened.

The paper is organized as follows. The next section presents some
preliminary results that will be used throughout the paper. We also discuss here our standing assumptions for the problem which is somewhat natural for the lack of Lipschitz assumption  aforementioned. Section \ref{search} devotes to the two different
linesearches for the forward-backward methods used in Sections \ref{s:3} and \ref{Sec-new-4}. Weak convergence  and complexity of the
forward-backward method with the first linesearch are analyzed in Section \ref{s:3}. We also consider its accelerated version here.
Section \ref{Sec-new-4} provides a similar study for a variant of  the forward-backward splitting method with the  second linesearch. We complete the paper with some  conclusion for further study.

\section{Preliminary results}\label{prelim}
In this section we present some definitions and results needed for
our paper. Let $h:\HH\to \overline{\RR}$ be a proper, lower
semicontinuous (\emph{l.s.c.}), and convex function. We denote the
domain of $h$ by $\dom h:=\{x\in \HH\,|\; h(x)<+\infty\}$. For any
$x\in \dom h$, the directional derivative of $h$ at $x$ in the
direction $d$ is
$$\displaystyle
h^\prime(x;d) := \lim_{t \rightarrow 0^+} \frac{h(x + td) -
h(x)}{t},
$$
which always exists (although it may be infinite). The
subdifferential of $h$ at $x$ is defined by
\begin{equation}\label{sub-inq}
\partial h(x):=\{v\in \HH\,|\; \la v,y-x\ra\le h(y)-h(x),\; y\in \HH\}.
\end{equation}
\begin{fact}[{\cite[Proposition~17.2]{librobauch}} ]\label{prop-cons-g-convex}  Let $h: \HH\to\overline{\RR}$ be a proper, l.s.c., and convex function. Then, for $x\in \dom h$ and $y\in \HH$, the following hold:
\item {{\bf (i)}}  $h^\prime(x;y)$ exists and
$\dsty h^\prime(x;y)= \inf_{t\in \RR_{++}} \frac{h(x+t
y)-h(x)}{t}. $
\item {{\bf (ii)}} $h^\prime(x;y-x)+h(x)\le h(y).$
\end{fact}
\begin{fact}[{\cite[Theorem~4.7.1 and Proposition~4.2.1(i)]{iusem-regina}} ]\label{teo-p1}
The subdifferential operator $\partial h$ is maxi\-mal monotone,
i.e., it has no proper monotone extension in the graph
inclusion sense. Moreover, the graph of $\partial h$, ${\rm
Gph}(\partial h):=\{(x,v)\in\HH\times \HH\,|\; v\in\partial h(x)\}$
is demiclosed, i.e., if the sequence $(x^k,v^k)_{k\in
\NN}\subset {\rm Gph}(\partial h)$ satisfies that  $(x^k)_{k\in\NN}$
converges weakly to $x$ and $(v^k)_{k\in\NN}$ converges strongly to
$v$, then $(x,v)\in {\rm Gph }(\partial h)$.
\end{fact}
\noindent Next we set  the standing assumptions on the data of problem
\eqref{prob} used throughout the paper as follows:
\begin{enumerate}[leftmargin=0.3in, label={\bf A\arabic*}]
\item\label{a1} $f,g:\HH\to\overline{\RR}$ are two proper
\emph{l.s.c.} convex functions with $\dom g\subseteq \dom f$.
\item\label{a2} The function $f$ is Fr\'echet differentiable on an open set containing $\dom g$.
The gradient $\nabla f$  is uniformly continuous on any
bounded subset of $\dom g$ and maps any bounded subset of $\dom g$ to a bounded set in $\HH$.
\end{enumerate}
Assumption \ref{a1} and the first part of Assumption \ref{a2} are popular and crucial for the well-definedness
of the forward-backward iteration \eqref{F-B-I}. It
is easy to check that the second part of \ref{a2} is automatic when $\nabla f$ is Lipschitz
continuous on $\dom g$. However, Assumption~{\bf A2} is not enough to guarantee the Lipschitz continuity of $\nabla f$.  Indeed,  the convex functions
$f(x)\equiv\|x\|^{p}$ ($1<p<+\infty$, $p\neq2$) and $g(x)\equiv 0$,  $x\in \HH$ satisfy
all the conditions in  \ref{a2} but $\nabla f$ is not globally
Lipschitz continuous. When $\HH$ is a finite-dimensional space and
the domain of $g$  is closed, Assumption \ref{a2} actually means
that $f$ is Fr\'echet differentiable on  an open set containing $\dom g$ and that its gradient is continuous on $\dom g$.
It is worth noting further that the closedness of $\dom g$ is broadly assumed for problem \eqref{prob} in the literature
 including the case of optimization problems with geometric constraints,  which can be written as \eqref{prob} when  $g$ is an indicator function; see,
\emph{e.g.}, \cite{neal-boyd}.

\begin{proposition}\label{pop} Let $\HH$ be a finite-dimensional space and let $f, g :\HH \rightarrow\,\overline{\RR}$ be two functions satisfying \ref{a1}.
Suppose that  the closure of $\dom g$, denoted by ${\rm cl}\, (\dom g)$ is a subset of  $\dom f$, $f$ is Fr\'echet differentiable on  an open set containing ${\rm cl}\, (\dom g)$, and that its gradient $\nabla f$ is continuous on ${\rm cl}\, (\dom g)$. Then Assumption \ref{a2} is satisfied.

Consequently, if $\dom g$ is closed then the validity of Assumption \ref{a2} is equivalent to the statement that  $f$ is Fr\'echet differentiable on  an open set containing $\dom g$ and its gradient is continuous on $\dom g$.
\end{proposition}
\begin{proof} To justify, suppose that $\dim \HH<+\infty$, ${\rm cl}\, (\dom g)\subseteq X\subseteq \dom f$, $f$ is Fr\'echet differentiable on  an open set containing ${\rm cl}\, (\dom g)$, and that $\nabla f$ is continuous on ${\rm cl}\, (\dom g)$.
Take any bounded set $A$ of $\dom g$. Note that  $\nabla f$ is uniformly
continuous on the compact set ${\rm cl}\, A\subseteq {\rm cl}\, (\dom g)$ and thus on
$A$ due to the classical \emph{Heine-Cantor Theorem}.
Since $\nabla f$ is continuous on ${\rm cl}\, (\dom g)$, it maps the
compact set ${\rm cl}\, A\subseteq {\rm cl}\, (\dom g)$ to a compact set in $\HH$.
This verifies that $\nabla f(A)$ is bounded and completes the first
part of the proposition.

Now suppose that $\dom g$ is closed. It is easy to see that the validity of Assumption \ref{a2} implies that $\nabla f$ is continuous on $\dom g$.  This together with the first part of this proposition justifies the second part. The proof is completed.
\end{proof}
Let us recall the proximal operator
$\prox_g:\HH\to \dom g$ with $\prox_g(z)=(\Id+\partial g)^{-1}(z)$, $z\in \HH$. It is well-known that the proximal operator is single-valued with full domain. Furthermore, note that
\begin{equation}\label{in-sub}
\frac{z-\prox_{\alpha g}(z)}{\al}\in \partial g(\prox_{\al g}(z))\quad\mbox{for all} \quad z\in\HH,\, \al\in  \RR_{++}:=\{t\in\RR|\; t>0\}.
\end{equation}
We also denote the {\em forward-backward operator} $J:{\dom
g}\times \RR_{++}\to \dom g\subset \HH$   by
\begin{equation}\label{J}
J(x,\alpha):=\prox_{\alpha g}(x-\alpha \nabla f(x))\quad \mbox{for
all}\quad x\in {\dom
g \subseteq} \;\dom f,\,\al>0.
\end{equation}
The following lemma is very useful for our further study.
\begin{lemma}[{{\cite[Lemma 1]{Dong1}}}]\label{lema-prop}
Let $f, g :\HH \rightarrow\, \overline{\RR}$ be two functions
satisfying Assumption \ref{a1}. Then for any $x\in {\dom
	g}$ and
$\alpha_2\ge \alpha_1>0$, we have
\begin{equation}\label{ineq-1}
\frac{\alpha_2}{\alpha_1}
\|x-J(x,\alpha_1)\|\ge\|x-J(x,\alpha_2)\|\ge \|x-J(x,\alpha_1)\|.
\end{equation}
\end{lemma}
%\begin{proof}
%By using \eqref{in-sub} with $z=x-\alpha \nabla f(x)$ and $x\in \dom
%f$, we have
%\begin{equation}\label{inc-1}
%\frac{x-\alpha \nabla f(x)-J(x,\alpha)}{\alpha}\in \partial g(J(x,\alpha))
%\end{equation} for all $\alpha>0$. Take any $\alpha_2\ge \alpha_1>0$, it follows from the monotonicity of $\partial g$ and \eqref{inc-1} that
%\begin{align*}
%0\le &\left \la \frac{x-\alpha_2 \nabla
%f(x)-J(x,\alpha_2)}{\alpha_2}-\frac{x-\alpha_1 \nabla
%f(x)-J(x,\alpha_1)}{\alpha_1},J(x,\alpha_2)-J(x,\alpha_1)\right\ra\\
%=&\left\la
%\frac{x-J(x,\alpha_2)}{\alpha_2}-\frac{x-J(x,\alpha_1)}{\alpha_1},\left(x-J(x,\alpha_1)\right)-\left(x-J(x,\alpha_2)\right)\right\ra\\
%=&-\frac{\|x-J(x,\alpha_2)\|^2}{\alpha_2}-\frac{\|x-J(x,\alpha_1)\|^2}{\alpha_1}+\left(\frac{1}{\alpha_2}+\frac{1}{\alpha_1}\right)\la
%x-J(x,\alpha_2),x-J(x,\alpha_1)\ra\\ \le&
%-\frac{\|x-J(x,\alpha_2)\|^2}{\alpha_2}-\frac{\|x-J(x,\alpha_1)\|^2}{\alpha_1}+\left(\frac{1}{\alpha_2}+\frac{1}{\alpha_1}\right)\|x-J(x,\alpha_2)\|\cdot\|x-J(x,\alpha_1)\|,
%\end{align*}
%which imply the following expression
%$$
%\big(\|x-J(x,\alpha_2)\|-\|x-J(x,\alpha_1)\|\big)\cdot\Big(\|x-J(x,\alpha_2)\|-\frac{\al_2}{\al_1}\|x-J(x,\alpha_1)\|\Big)\le 0.
%$$
%Since $\dsty\frac{\alpha_2}{\alpha_1}\ge 1$, we derive
%\eqref{ineq-1} and thus complete the proof of the lemma.
%\end{proof}
Let us end the section by recalling the well-known concepts of
so-called quasi-Fej\'er and Fej\'er convergence. The definition
originates in \cite{Ermolev} and has been elaborated further in
\cite{IST,comb-2001}.
\begin{definition}\label{def-fejer}
Let $S$ be a nonempty subset of $\HH$. A sequence $(x^k)_{k\in \NN}$
in $\HH$ is said to be quasi-Fej\'er convergent to $S$ if and only
if for all $x \in S$ there exists a sequence $(\epsilon_k)_{k\in
\NN}$ in $\RR_+$ such that $\sum_{k=0}^\infty\epsilon_k<+\infty$ and
$\| x^{k+1}-x\|^2 \leq \| x^{k}-x\|^2 +\epsilon_k$ for all
$k\in\NN$. When $(\epsilon_k)_{k\in \NN}$ is a null sequence, we say
that $(x^k)_{k\in \NN}$ is Fej\'er convergent to $S$.
\end{definition}
\begin{fact}[{\cite[Theorem~4.1]{IST}} ]\label{lema-Fejer} If
$(x^k)_{k\in\NN}$ is quasi-Fej\'er convergent to $S$, then one has:
\item {\bf (i)} The sequence $(x^k)_{k\in\NN}$ is bounded.

\item {\bf (ii)} If all weak accumulation points of $(x^k)_{k\in\NN}$ belong to $S$,
then $(x^k)_{k\in\NN}$ is weakly convergent to a point in $S$.
\end{fact}
\section{The linesearches}\label{search}
In this section we present two different linesearches mainly used in
the forward-backward methods proposed in Sections \ref{s:3}
and \ref{Sec-new-4}. The first one contains a backtracking procedure
 which computes {\em at least one} backward step (implicit
step) inside the updating inner loop for finding the steplength.
This linesearch is a particular case of the one proposed in
\cite{tseng} for solving inclusion problems. It will be used in {\bf
Method~\ref{A1}} and {\bf Method \ref{A3}} of Section~\ref{s:3}.
\begin{center}\fbox{\begin{minipage}[b]{\textwidth}
\begin{Line}\label{boundary} Given $x$, $\sigma>0$, $\theta\in (0,1)$
and $\delta\in (0,1/2)$.\\
{\bf Input.} Set $\alpha=\sigma$ and $J(x,\alpha):=\prox_{\alpha
g}(x-\alpha \nabla f(x))$ with $x\in \dom g$.
\begin{retraitsimple}
\item[] {\bf While} $
\alpha\left\|\nabla f\big(J(x,\alpha)\big)-\nabla
f(x)\right\|>\delta\left\|J(x,\alpha)-x\right\|$   {\bf do}\\
$\alpha=\theta \alpha$.
\item[] {\bf End While}
\end{retraitsimple}
{\bf Output.} $\alpha$.
\end{Line}\end{minipage}}\end{center}
The well-definedness of {\bf Linesearch \ref{boundary}} follows from
\cite[Theorem~3.4(a)]{tseng}. For the reader's convenience, we
provide a different proof revealing that  the convexity of $f$ is
not necessary.
\begin{lemma}\label{boundary-well} If $x\in \dom g$ then {\bf Linesearch~\ref{boundary}} stops after finitely many steps.
\end{lemma}
\begin{proof} If $x\in S_*$ then $x=J(x,\sigma)$. Thus the linesearch stops with zero step and gives us the output $\sigma$.
If $x\notin S_*$, by contradiction suppose that  for all $\alpha\in\PP:=\{\sigma, \sigma\theta, \sigma\theta^2, \ldots \}$,
\begin{equation}\label{no-conda}
\alpha\left\|\nabla f\big(J(x,\alpha)\big)-\nabla f(x)\right\|>\delta\left\|J(x,\alpha)-x\right\|.
\end{equation}
When $\alpha\in \PP$ is sufficiently closed to $0$, it follows from
Lemma~\ref{lema-prop} that $J(x,\al)$ is uniformly bounded. Thus we
get from  \eqref{no-conda} that $ \|x-J(x,\alpha)\|\to 0$ as
$\al\downarrow 0$ thanks to Assumption \ref{a2}. The latter implies
$\|\nabla f\big(J(x,\alpha)\big)-\nabla f(x)\|\to0$ when
$\al\downarrow  0$ by Assumption \ref{a2} again. Thus we get from \eqref{no-conda}  that
\begin{equation}\label{eq-1.1}
\lim_{\alpha\downarrow  0}\frac{\|x-J(x,\alpha)\|}{\alpha}=0.
\end{equation}
Employing  \eqref{in-sub} with $z=x-\alpha \nabla f(x)$ gives us that
$$\displaystyle
\frac{x-J(x,\alpha)}{\alpha}\in \nabla
f(x)+\partial g(J(x,\alpha)).
$$
 By letting  $\alpha\downarrow 0$ in the above inclusion and using \eqref{eq-1.1}, we get from the demiclosedness of ${\rm Gph}(\partial g)$  from Fact \ref{teo-p1} that $0\in \nabla f(x)+\partial g(x)\subseteq \partial (f+g)(x)$.
This contradicts the assumption that $x$ is not an optimal solution to problem \eqref{prob} and completes the proof of the lemma.
\end{proof}

Next we propose the second backtracking procedure. In contrast to
{\bf Linesearch 1}, this linesearch  demands {\em only one}
evaluation of the backward step and uses it in all possible
iterations. This is somehow an advantage of this linesearch, since
in many practical problems computing the proximal operator many
times may be very expensive. The linesearch is indeed a
generalization of  the one studied in \cite{yun-wlo} for solving the
nonlinear constrained optimization problem ($g=\delta_C$). We will
employ it in {\bf Method~\ref{A2}} in Section \ref{Sec-new-4}.

\begin{center}\fbox{\begin{minipage}[b]{\textwidth}
\begin{Line}\label{boundary2} Given $x$ and $\theta\in (0,1)$.\\
{\bf Input.} Set $\beta=1$, $J_x:=J(x,1)=\prox_{g}(x-\nabla f(x))$ with $x\in \dom g$.
\begin{retraitsimple}
\item[] {\bf While} $\dsty (f+g)\left(x -\beta (x-J_x)\right) > (f+g)(x)-\beta\left[g(x)-g(J_x)\right]-\beta
\la\nabla f(x), x-J_x\ra+\frac{\beta}{2}\|x-J_x\|^2$   {\bf do}\\
$\beta=\theta \beta$.
\item[] {\bf End While}
\end{retraitsimple}
{\bf Output.} $\beta$.
\end{Line}\end{minipage}}\end{center}
Similarly to {\bf Linesearch \ref{boundary}}, we also have finite
termination for {\bf Linesearch \ref{boundary2}}. It is
important to note that the well-definedness analysis is done without
assuming the second part of {\bf \ref{a2}} (uniform continuity and
boundedness).
\begin{lemma}\label{boundary-well2} If $x\in \dom g$ then {\bf Linesearch~\ref{boundary2}} stops after finitely many steps.
\end{lemma}
\begin{proof}
If $x\in S_*$ we have $x=J_x$. Thus the linesearch immediately gives
us the output $1$ without proceeding any step. If $x\notin S_*$, by
contradiction let us assume that {\bf Linesearch \ref{boundary2}}
does not stop after  finitely many steps. Thus for all $\beta\in
\QQ:=\{1, \theta, \theta^2, \ldots \}$, we have
\begin{equation*}\label{no12*}
(f+g)(x -\beta (x-J_x)) > (f+g)(x)-\beta\left[g(x)-g(J_x)\right]-\beta
\la\nabla f(x), x-J_x\ra+\frac{\beta}{2}\|x-J_x\|^2.
\end{equation*}
It follows that
$$\displaystyle
\frac{(f+g)(x -\beta (x-J_x)) - (f+g)(x)}{\beta} +
g(x)-g(J_x)+\la\nabla f(x), x-J_x\ra
>\frac{1}{2}\|x-J_x \|^2.
$$
Taking  $\beta\downarrow 0$  and using the Fr\'echet
differentiability of $f$ and the convexity of $g$ give us that
\begin{align*}
\frac{1}{2}\|x-J_x \|^2\le& \la\nabla f(x), J_x-x\ra+
g^\prime(x;J_x-x)+ g(x)-g(J_x)+\la\nabla f(x), x-J_x\ra\\
=&g^\prime(x;J_x-x)+ g(x)-g(J_x) \le 0,
\end{align*} where the last inequality follows from Fact \ref{prop-cons-g-convex}{(ii)}.
Hence we have  $x=J_x$, which readily implies that $x-\nabla f(x)\in
\partial g(x)+x$, \emph{i.e.}, $0\in \nabla f(x)+\partial g(x)\subseteq \partial(f+g)(x)$. This contradicts the assumption $x\not\in S_*$.
\end{proof}
%%%%%%%%%%%%%%%%%%%%%%%%%%%%%%%%%%%%%%%%%%%%%%%%%%%%%%%%%%%%%%%%%%%%%%%%%%%%%%%%%%%%
%%%%%%%%%%%%%%%%%%%%%%%%%%%%%%%% Algorithm 1 %%%%%%%%%%%%%%%%%%%%%%%%%%%%%%%%%%%%%%%
%%%%%%%%%%%%%%%%%%%%%%%%%%%%%%%%%%%%%%%%%%%%%%%%%%%%%%%%%%%%%%%%%%%%%%%%%%%%%%%%%%%%
\section{The forward-backward method with Linesearch \ref{boundary}}\label{s:3}
This section devotes to the study of  the forward-backward
splitting method with {\bf Linesearch~\ref{boundary}}. We
mainly derive the weak convergence of the generated sequences from
this method and also obtain the same complexity of \cite[Theorem
$1.1$]{beck-teu} for the cost value sequences generated from  the forward-backward
iteration under a weaker assumption than the Lipschitz one on
$\nabla f$ usually imposed in the literature.

The following method has some similarities to the one proposed in \cite{tseng} for maximal monotone operators. However, it completely relaxes an extra expensive projection step \cite[Equation~(2.3)]{tseng}  and seems to be more natural in comparison with the classical forward-backward splitting method \eqref{F-B-I}.
\begin{center}\fbox{\begin{minipage}[b]{\textwidth}
\begin{method}\label{A1}

\item [    ] {\bf Initialization Step.} Take $x^0\in \dom g$, $\sigma>0$, $\theta\in(0,1)$ and $\delta\in(0,1/2)$.

\item [    ] \noindent {\bf Iterative Step.} Given $x^k$ set
\begin{equation}\label{F-B}
x^{k+1}=J(x^k,\alpha_k):=\prox_{\alpha_k g}(x^k-\alpha_k\nabla
f(x^k))
\end{equation}
with $\alpha_k :=$ {\bf
Linesearch~\ref{boundary}}$(x^k,\sigma,\theta,\delta)$.

\item [    ] \noindent  {\bf Stop Criteria.} If $x^{k+1}=x^k$, then stop.
\end{method}\end{minipage}}\end{center}
First note that from Lemma~\ref{boundary-well} that {\bf Linesearch~\ref{boundary}} for finding the stepsize
$\alpha_k$ in the above scheme is finite. Hence the choice of
sequence $(x^k)_{k\in \NN}$ in {\bf Method~\ref{A1}} is well-defined.
Another important feature from the definition of {\bf
Linesearch~\ref{boundary}} useful for our analysis is the following
inequality
\begin{equation}\label{Li}
\alpha_k\left\|\nabla f(x^{k+1})-\nabla
f(x^k)\right\|\le \delta\left\|x^{k+1}-x^{k}\right\|.
\end{equation}
Note further that if {\bf Method~\ref{A1}} stops  at iteration $k$
then we have $x^k=\prox_{\alpha_k g}(x^k-\alpha_k\nabla f(x^k))$ and
 consequently  $x^k\in S_*$.  Otherwise,  we will mainly show that the sequence $(x^k)_{k\in \NN}$
generated by this method is converging weakly to some optimal
solution. Verifying this claim needs some auxiliary results as
follows.
\begin{proposition}\label{prop1} Let $\alpha_k={\bf Linesearch~\ref{boundary}}(x^k,\sigma,\theta,\delta)$. For all $k\in
\NN$ and $x\in\dom g$, we have
\item {\bf (i)} $\|x^k-x\|^2-\|x^{k+1}-x\|^2\ge 2\alpha_k\left[(f+g)(x^{k+1})-(f+g)(x)\right]+(1-2\delta)\|x^{k+1}-x^k\|^2$;
\item {\bf(ii)} $(f+g)(x^{k+1})-(f+g)(x^k)\le - \frac{\dsty(1-\delta)}{\dsty \alpha_k}\|x^{k+1}-x^k\|^2.$
\end{proposition}
\begin{proof}
First let us justify {\bf (i)} by noting from \eqref{in-sub} and \eqref{F-B} that
\begin{equation*}
\frac{x^k-x^{k+1}}{\alpha_k}-\nabla f(x^k)=\frac{x^k-J(x^k,\al_k)}{\alpha_k}-\nabla f(x^k)\in\partial g(J(x^k,\al_k))=\partial g(x^{k+1}).
\end{equation*}
It follows from the convexity of $g$ that
\begin{equation}\label{eq2}g(x)-g(x^{k+1})\ge \left\la\frac{x^k-x^{k+1}}{\alpha_k}-\nabla f(x^k),x-x^{k+1}\right\ra\quad\mbox{for all}\quad x\in \dom g.
\end{equation}
Since $f$ is convex, we also have
\begin{equation}\label{eq3}
f(x)-f(y)\ge \la\nabla f(y), x-y\ra\quad\mbox{for all}\quad x\in \dom f, \; \; y\in {\dom
	g}.
\end{equation}
Summing \eqref{eq2} and \eqref{eq3} with  any $x\in \dom g\subseteq \dom f$ and $y=x^{k}\in \dom g$ gives us the following expressions
\begin{align*}\nonumber\label{eq4}(f+g)(x)\ge& f(x^k)+g(x^{k+1}) + \left\la\frac{x^k-x^{k+1}}{\alpha_k}-\nabla f(x^k),x-x^{k+1}\right\ra+\la\nabla f(x^k), x-x^k\ra\\
\nonumber=& f(x^k)+g(x^{k+1}) +\frac{1}{\alpha_k} \la
x^k-x^{k+1},x-x^{k+1}\ra+\la\nabla f(x^k), x^{k+1}-x^k\ra\\
\nonumber\ge&f(x^k)+g(x^{k+1}) +\frac{1}{\alpha_k} \la x^k-x^{k+1},x-x^{k+1}\ra+\la\nabla f(x^{k+1}),x^{k+1}-x^{k}\ra\\
\nonumber&-\|\nabla f(x^k)-\nabla f(x^{k+1})\|\cdot\|x^{k+1}-x^k\|\\
\ge&f(x^k)+g(x^{k+1}) +\frac{1}{\alpha_k} \la
x^k-x^{k+1},x-x^{k+1}\ra+\la\nabla f(x^{k+1}),x^{k+1}-x^{k}\ra\\&-\frac{\delta}{\alpha_k}\|
x^{k+1}-x^k\|^2,
\end{align*}
where the last inequality follows from \eqref{Li}. After rearrangement we get
\begin{equation}\label{eq5}\la x^k-x^{k+1},x^{k+1}-x\ra\ge \alpha_k[f(x^k)+g(x^{k+1})-(f+g)(x)+
\la\nabla f(x^{k+1}),x^{k+1}-x^{k}\ra]-\delta\|
x^{k+1}-x^k\|^2.
\end{equation}
Since
$
2\la x^k-x^{k+1},x^{k+1}-x \ra=\|x^k-x\|^2-\|x^{k+1}-x\|^2-\|x^k-x^{k+1}\|^2,
$
we get from  \eqref{eq5} that
\begin{eqnarray}\begin{array}{ll}
\|x^k-x\|^2-\|x^{k+1}-x\|^2\ge &2\alpha_k[f(x^k)+g(x^{k+1})-(f+g)(x)]\\
&\disp+2\alpha_k\la\nabla f(x^{k+1}),x^{k+1}-x^k\ra
+(1-2\delta)\|x^k-x^{k+1}\|^2.\label{eq6}
\end{array}
\end{eqnarray}
 By using \eqref{eq3} with $x=x^k$ and
$y=x^{k+1}$, we  have $f(x^k)-f(x^{k+1})\ge\la\nabla
f(x^{k+1}),x^k-x^{k+1}\ra.$ This together with \eqref{eq6} gives us
that
\begin{align*}\|x^k-x\|^2-\|x^{k+1}-x\|^2\ge&
2\alpha_k[(f+g)(x^{k+1})-(f+g)(x)]+(1-2\delta)\|x^k-x^{k+1}\|^2,
\end{align*}
which  verifies {\bf (i)}. Note further that {\bf (ii)} is a consequence of {\bf (i)} when $x=x^k$. The proof  is complete.
\end{proof}

Proposition \ref{prop1}(ii) shows that {\bf Method~\ref{A1}} is
a descent method in the sense that the value of the cost function
$f+g$ at each iteration is decreasing. Furthermore, it is easy to
check from  Proposition \ref{prop1}(i) that the generated sequence
of {\bf Method~\ref{A1}} is Fej\'er convergent to the optimal
solution set $S_*$ whenever $S_*\neq\emptyset$. This observation is
indeed the center of the following main result of this section, where we prove the weak convergence of sequence $(x^k)_{k\in \NN}$ in {\bf Method~\ref{A1}} and also  $\left((f+g)(x^k)\right)_{k\in \NN}$ is a minimizing sequence of $f+g$ without the Lipschitz assumption on $\nabla f$.  To the best of our knowledge, this result improves \cite[Theorem~1.2]{beck-teu} and even the classical results for gradient method with linesearch; see, for instance, \cite[Proposition 1.3.3]{bertsekas } and \cite{Armijo-1}. Moreover, we show that the  sequence $\left((f+g)(x^k)\right)_{k\in \NN}$ converges to the infimum value when the solution set is empty. 

\begin{theorem}\label{new-cov}  Let $(x^k)_{k\in \NN}$ and $(\al_k)_{k\in\NN}$ be the sequences generated by {\bf Method~\ref{A1}}. The following statements hold:

\item {\bf (i)} If $S_*\neq\emptyset$ then $(x^k)_{k\in \NN}$ is  weakly convergent to a point in
$S_*$. Moreover, \begin{equation}\label{min*}\lim_{k\to\infty}
(f+g)(x^k)=\min_{x\in \HH} \, (f+g)(x).\end{equation}

\item {\bf (ii)} If $S_*=\emptyset$ then we have
\begin{equation*}\label{inf}
\lim_{k\to\infty}\|x^k\|=+\infty \quad \mbox{and}
\quad\lim_{k\to\infty} (f+g)(x^k)=\inf_{x\in \HH}(f+g)(x).
\end{equation*}
\end{theorem}
\begin{proof} Let us justify {\bf (i)} by supposing that $S_*\neq\emptyset$. By applying Proposition~\ref{prop1}(i) at any $x_*\in S_*$, we have
\begin{align}\label{in15}\|x^k-x_*\|^2-\|x^{k+1}-x_*\|^2&\disp\ge
2\alpha_k[(f+g)(x^{k+1})-(f+g)(x_*)]+(1-2\delta)\|x^k-x^{k+1}\|^2\\
&\disp\ge (1-2\delta)\|x^k-x^{k+1}\|^2\ge0.\nonumber
\end{align}
It follows that  the
sequence $(x^k)_{k\in\NN}$ is Fej\'er convergent to
$S_*$ and thus is
bounded by Fact~\ref{lema-Fejer}(i). By using
\eqref{in15}, we get
\begin{align*}\nonumber 
0&\le2\alpha_k[(f+g)(x^{k+1})-(f+g)(x_*)]\le
\|x^k-x_*\|^2-\|x^{k+1}-x_*\|^2\\\label{paracero}
&=(\|x^k-x_*\|+\|x^{k+1}-x_*\|)\cdot(\|x^k-x_*\|-\|x^{k+1}-x_*\|)\\\nonumber
&\le
2M(\|x^k-x_*\|-\|x^{k+1}-x_*\|)\\\nonumber
&\le 2M\|x^k-x^{k+1}\|,
\end{align*} where $M:=\sup\{\|x^k-x_*\||\; k\in\NN\}<+\infty$. Hence the above inequalities lead us to
\begin{equation}\label{like17}
(f+g)(x^{k+1})-(f+g)(x_*)\le
M\,\frac{\|x^k-x^{k+1}\|}{\al_k}.
\end{equation}
Due to the Fej\'er property of $(x^k)_{k\in\NN}$ to $S_*$, the sequence $(\|x^k-x^*\|)_{k\in \NN}$ is convergent. This together with \eqref{in15} tells us that $\|x^k-x^{k+1}\|\to 0$ as $k\to\infty$.

Since $(x^k)_{k\in\NN}$  is bounded, the set of its weak
accumulation points is nonempty. Take any weak
accumulation point  $\bar x$ of $(x^k)_{k\in\NN}$, we find a subsequence
$(x^{n_k})_{k\in\NN}$ weakly converging to $\bar x$. Now let us
split our further analysis into two distinct cases.

\medskip

\noindent {\bf Case 1.} Suppose that the sequence $(\alpha_{n_k})_{k\in
\NN}$ defined in {\bf Method~\ref{A1}} does not converge to $0$. Hence there exist a subsequence (without relabelling) of $(\alpha_{n_k})_{k\in
\NN}$ and  $\alpha>0$
such that
\begin{equation}\label{bk-no-0**}
{\alpha_{n_k}\geq\alpha.}
\end{equation}
Since $(x^k)_{k\in\NN}$ is bounded and $\|x^k-x^{k+1}\|\to 0$ as
claimed above, we get from Assumption \ref{a2} that
\begin{equation}\label{grad-to-0**}
\dsty\lim_{k\rightarrow\infty}\| \nabla f(x^{n_k})-\nabla
f(x^{n_k+1})\| =0.
\end{equation}
Since $x^{n_k+1}=J(x^{n_k},\alpha_{n_k})$, it follows from \eqref{in-sub} and \eqref{F-B} that
$$\displaystyle
\frac{x^{n_k}-\al_{n_k}\nabla f(x^{n_k})-x^{{n_k}+1}}{\alpha_{n_k}} \in
 \partial g(x^{n_k+1}),
$$
which implies in turn the expression
\begin{eqnarray}\label{abc}
\frac{x^{n_k}-x^{{n_k}+1}}{\alpha_{n_k}} + \nabla f(x^{n_k+1})-\nabla
f(x^{n_k})\in
\nabla f(x^{n_k+1})+\partial g(x^{n_k+1})\subseteq \partial (f+g)(x^{n_k+1}).
\end{eqnarray}
Note also that the subsequence $(x^{n_k+1})_{k\in\NN}$ converges weakly to $\bar x$ due to the fact that $\|x^{n_k}-x^{n_k+1}\|\to 0$ as $k\to\infty$. By passing $k\to \infty$ in \eqref{abc}, we get from
\eqref{bk-no-0**}, \eqref{grad-to-0**}, and  Fact \ref{teo-p1}  that
$ 0\in \partial (f+g)(\bar{x})$, which means $\bar{x}\in
S_*$. Furthermore, since the sequence $((f+g)(x^k))_{k\in\NN}$ is decreasing due to Proposition~\ref{prop1}(ii), \eqref{min*} is a consequence of \eqref{like17} and \eqref{bk-no-0**}.

\medskip

\noindent{\bf Case 2.} Suppose now 
$\lim_{k\rightarrow\infty}\alpha_{n_k}=0$. Define
$\hat{\alpha}_{n_k}:=\dsty
\frac{\alpha_{n_k}}{\theta}>\alpha_{n_k}>0$ and
$\hat{x}^{n_k}:=J\left(x^{n_k},\hat{\alpha}_{n_k}\right)$. Due to
Lemma~\ref{lema-prop} we have
\[
\|x^{n_k}-\hat x^{n_k}\|=\|x^{n_k}-J(x^{n_k},\hat \alpha_{n_k})\|\le \frac{\hat \al_{n_k}}{\al_{n_k}}\|x^{n_k}-J(x^{n_k},\al_{n_k})\|= \frac{1}{\theta}\|x^{n_k}-x^{n_k+1}\|,
\]
which combines with the boundedness of  $(x^{n_k})_{k\in \NN}$ to show
that the sequence  $(\hat x^{n_k})_{k\in \NN}$ is also bounded. It follows from the definition of {\bf Linesearch~\ref{boundary}} that
\begin{equation}\label{eq-no-armijo}
\hat{\alpha}_{n_k}\left\|\nabla f\big(\hat{x}^{n_k}\big)-\nabla
f(x^{n_k})\right\|>\delta\left\|\hat{x}^{n_k}-x^{n_k}\right\|.
\end{equation}
Since $\hat{\alpha}_{n_k}\downarrow 0$ and  both $(x^{n_k})_{k\in
\NN}$ and $(\hat x^{n_k})_{k\in \NN}$ are bounded,
\eqref{eq-no-armijo} together with Assumption
\ref{a2}  tells us that
$\lim_{k\rightarrow\infty}\|\hat{x}^{n_k}-x^{n_k}\|=0$  and thus  $(\hat{x}^{n_k})_{k\in \NN}$ also weakly converges to $\bar x$. Thanks to
Assumption \ref{a2} again,  we have
\begin{equation}\label{hatgradk}\lim_{k\rightarrow\infty}\left\|\nabla f\big(\hat{x}^{n_k}\big)-\nabla
f(x^{n_k})\right\|=0.\end{equation}
This and   \eqref{eq-no-armijo} imply that
\begin{equation}\label{hatxk}
\lim_{k\rightarrow\infty}\frac{1}{\hat{\alpha}_{n_k}}\|\hat{x}^{n_k}-x^{n_k}\|=0.\end{equation}
Using  \eqref{in-sub} with $z=x^{n_k}-\hat{\alpha}_{n_k} \nabla f(x^{n_k})$ gives us that
\begin{equation*}
\frac{x^{n_k}-\hat \al_{n_k}\nabla f(x^{n_k})-\hat{x}^{n_k}}{\dsty\hat{\alpha}_{n_k}} + \nabla
f(\hat{x}^{n_k})\in \partial
g(\hat{x}^{n_k})+\nabla f(\hat{x}^{n_k}) \subseteq\partial (f+g)(\hat{x}^{n_k}).
\end{equation*}
By letting $k\to\infty$, we get from the latter, \eqref{hatgradk}, 
\eqref{hatxk}, and Fact~\ref{teo-p1} that $ 0\in
\partial (f+g)(\bar{x})$, which  means  $\bar{x}\in S_*$.
 It remains to verify \eqref{min*}
in this case. Indeed, we get from Lemma~\ref{lema-prop} that
\[
\|x^{n_k}-\hat
x^{n_k}\|=\|x^{n_k}-J(x^{n_k},\frac{\al_{n_k}}{\theta})\|\ge
\|x^{n_k}-J(x^{n_k},\al_{n_k})\|=\|x^{n_k}-x^{n_k+1}\|.
\]
This together with \eqref{hatxk} yields $\frac{\|x^{n_k}-x^{n_k+1}\|}{\al_{n_k}}\to 0$ as $k\to \infty$. Since $(f+g)(x^k)$ is decreasing due to Proposition~\ref{prop1}(ii), we derive from the latter and \eqref{like17} that
\[
0=\lim_{k\to\infty} M\,\frac{\|x^{n_k}-x^{{n_k}+1}\|}{\al_{n_k}}\ge \lim_{k\to\infty}(f+g)(x^{n_k})-(f+g)(x_*)=\lim_{k\to\infty}(f+g)(x^k)-(f+g)(x^*)\ge 0,
\]
which clearly ensures \eqref{min*}.

From both cases above, we have \eqref{min*} and  the fact that any weak accumulation point of
$(x^k)_{k\in \NN}$  is an element of $S_*$. Thanks to
Fact~\ref{lema-Fejer}(ii), the sequence $(x^k)_{k\in \NN}$ weakly
converges to some point in $S_*$. This verifies {\bf (i)} of the
theorem.

To justify {\bf (ii)}, suppose that $S_*=\emptyset$. Observe  from the
proof of {\bf (i)} (without regarding \eqref{min*}, \eqref{in15}, and \eqref{like17}) that if $(x^k)_{k\in \NN}$ has any weak
accumulation point then this point is an optimal solution as illustrated in both cases there. Since $S_*=\emptyset$, any subsequence of $(x^k)_{k\in \NN}$  is unbounded
and thus $\|x^k\|\to+\infty$ as $k\to\infty$. Furthermore, note
that
$
s:=\lim_{k\to\infty} (f+g)(x^k)\ge \inf_{x\in \HH}(f+g)(x),
$
where $s$ exists due to fact that  $((f+g)(x^k))_{k\in\NN}$ is decreasing by Proposition~\ref{prop1}(ii). If $s>\inf_{x\in \HH}(f+g)(x)$ then  the
following auxiliary set
\begin{equation*} S_{\rm lev}(x^0):=\left\{x\in \dom g\colon
\;\; (f+g)(x)\le (f+g)(x^k),\; \forall k\in
\NN\right\}
\end{equation*}
is nonempty. By applying Proposition~\ref{prop1}(i) at any $x\in
S_{\rm lev}(x^0)$, similarly to \eqref{in15} we also have
$(x^k)_{k\in \NN}$ is F\'ejer convergent to $S_{\rm lev}(x^0)$. It
follows from  Fact~\ref{lema-Fejer}(i) that the sequence
$(x^k)_{k\in \NN}$ is bounded, which is a contradiction. Hence we
have $s=\inf_{x\in \HH}(f+g)(x)$ and complete the proof of the theorem.
\end{proof}

As discussed before {\bf Method~\ref{A1}}, our method improves the
scheme in \cite{tseng} for the particular case that  the two maximal monotone
operators considered there are $\nabla f$ and $\partial g$ by
relaxing completely  an additional step. Our
Theorem~\ref{new-cov} also loosens some unnatural assumptions
imposed in \cite[Theorem~3.4(b)]{tseng}. Furthermore,  we  obtain new information on the convergence of the cost values at generated sequences in this result.

\subsection{Complexity analysis of Method \ref{A1}}
In this subsection we present complexity analysis of the iterates in
{\bf Method~\ref{A1}}. When the stepsizes generated by {\bf
Linesearch \ref{boundary}} are bounded below by a positive
number,  our analysis shows that the expected error from the cost
value at the $k$-th iteration  to the optimal value is
$\mathcal{O}(k^{-1})$ in Hilbert spaces and $o(k^{-1})$ in finite dimensions, which improves the complexity of the first-order algorithm presented in \cite[Theorem~1.1]{beck-teu}. It is worth
emphasizing  that the global Lipschitz continuity assumption on the
gradient $\nabla f$ used in  \cite[Theorem~1.1]{beck-teu} is sufficient but not necessary for
the boundedness from below of the stepsizes aforementioned; see our
Proposition~\ref{lema-alpha} below. Since $\al_k>0$ for any $k\in
\NN$, this boundedness assumption actually  means that
$\liminf_{k\to\infty}\al_k>0$, which was used before in \cite{tseng}
for different purposes.
\begin{theorem}\label{l-rate}
Let $(x^k)_{k\in \NN}$ and $(\al_k)_{k\in \NN}$ be the sequences
generated in {\bf Method~\ref{A1}}. Suppose that $S_*\neq\emptyset$ and there exists
$\al>0$ such that $\alpha_k\ge\alpha>0$ for all $k\in \NN$. Then we have
\begin{equation}\label{rate}
(f+g)(x^k)-\min_{x\in \HH}\,(f+g)(x)\le\frac{1}{2\alpha } \frac{[{\rm
dist}(x^0,S_*)]^2}{k} \quad \mbox{for all}\quad k\in \NN.
\end{equation}
If in addition $\dim \HH<+\infty$ then 
\begin{equation}\label{xrate}
\lim_{k\to \infty}k\big[(f+g)(x^k)-\min_{x\in \HH}\,(f+g)(x)\big]=0. 
\end{equation}
\end{theorem}
\begin{proof}
Pick any $x_*\in S_*$,  Proposition \ref{prop1}(i) tells us that
\begin{align}\label{antes-sum}\nonumber
0\ge(f+g)(x_*)-(f+g)(x^{\ell+1})&\ge\frac{1}{2\alpha_\ell}\Big(
\|x^{\ell+1}-x_*\|^2-\|x^\ell-x_*\|^2+(1-2\delta)\|x^\ell-x^{\ell+1}\|^2\Big)\\
&\ge\disp  \frac{1}{2\alpha_\ell}( \|x^{\ell+1}-x_*\|^2-\|x^\ell-x_*\|^2)
\end{align}
for any $\ell\in \NN$. Since $\al_\ell\ge
\al$, we get from  \eqref{antes-sum} that
\begin{equation}\label{30}
0\ge(f+g)(x_*)-(f+g)(x^{\ell+1})\ge \frac{1}{2\alpha}(
\|x^{\ell+1}-x_*\|^2-\|x^\ell-x_*\|^2).
\end{equation}
 Summing the above
inequality over $\ell=0,1,\ldots,k-1$ implies that
\begin{equation*}
k(f+g)(x_*)-\sum_{\ell=0}^{k-1}(f+g)(x^{\ell+1})\ge\frac{1}{2\alpha}(
\|x^{k}-x_*\|^2-\|x^0-x_*\|^2).
\end{equation*}
Since $(f+g)(x^{\ell})$ is decreasing by Proposition~\ref{prop1}(ii), the latter  yields
\begin{equation}\label{dis}
k[(f+g)(x^k)-(f+g)(x_*)]\le\frac{1}{2\alpha }
\left(\|x_*-x^0\|^2-\|x^{k}-x_*\|^2\right)\le
\frac{1}{2\al}\|x_*-x^0\|^2.
\end{equation}
Note that no matter how we choose $x_*\in S$, the optimal
value $(f+g)(x_*)=\min_{x\in \HH}\,(f+g)(x)$ is fixed. Hence we get
from \eqref{dis} that
\begin{equation*}
(f+g)(x^k)-\min_{x\in \HH}\,(f+g)(x)\le\frac{1}{2\al}\inf_{y\in S_*}
\frac{\|y-x^0\|^2}{k}=\frac{1}{2\al}\frac{[{\rm
dist}(x^0,S_*)]^2}{k},
\end{equation*}
which verifies \eqref{rate} and completes the first part of the theorem.

Now suppose additionally that $\dim \HH<+\infty$, it follows from Theorem~\ref{new-cov} that the sequence $(x^k)_{k\in \NN}$
converges strongly  to some $x_*\in S_*$, \emph{i.e.}, $\|x^k-x_*\|\to 0$ as $k\to \infty$.  Take any $\ve>0$, we find $K\in \NN$ such that $\|x^k-x_*\|\le \ve$ for $k\ge K$.  For any $\ell\ge K$ we get from the fact $\|x^{\ell}-x_*\|\ge\|x^{\ell+1}-x_*\|$ and \eqref{antes-sum} that
 \begin{align}\label{32}\nonumber
0\ge(f+g)(x_*)-(f+g)(x^{\ell+1})&\ge \frac{1}{2\alpha_\ell}\big(
\|x^{\ell+1}-x_*\|+\|x^\ell-x_*\|\big)\cdot\big(
\|x^{\ell+1}-x_*\|-\|x^\ell-x_*\|\big)\\
&\ge \frac{1}{\alpha_\ell}\|x^\ell-x_*\|\big(
\|x^{\ell+1}-x_*\|-\|x^\ell-x_*\|\big)\\&\ge \frac{\ve}{\alpha}\big(
\|x^{\ell+1}-x_*\|-\|x^\ell-x_*\|\big).\nonumber
\end{align}
Now adding the above
inequality over $\ell=K,K+1,\ldots,K+k-1$ gives us  that
\begin{equation*}
 k(f+g)(x_*)-\sum_{\ell=K}^{K+k-1}(f+g)(x^{\ell+1})\ge\frac{\ve}{\alpha}(
\|x^{K+k}-x_*\|-\|x^K-x_*\|)\ge -\frac{\ve}{\alpha}\|x^K-x_*\|\ge  -\frac{\ve^2}{\alpha}.
\end{equation*}
Due to the decreasing property of  $(f+g)(x^{\ell})$ in Proposition~\ref{prop1}(ii), we get from the latter that 
\begin{equation*}
 k\big[(f+g)(x_*)-(f+g)(x^{K+k})\big]\ge -\frac{\ve^2}{\alpha}.
\end{equation*}
It follows that 
\begin{align*}
\limsup_{k\to\infty}\, k\big[(f+g)(x^{k})-(f+g)(x_*)\big]&=\limsup_{k\to\infty}\, (K+k)\big[(f+g)(x^{K+k})-(f+g)(x_*)\big]\\
&\le\limsup_{k\to\infty}\,\frac{K+k}{k}\cdot\frac{\ve^2}{\al}=\frac{\ve^2}{\al}. 
\end{align*}
Since this inequality holds for any $\ve>0$, we have 
\[
0\le\liminf_{k\to\infty}\, k\big[(f+g)(x^{k})-(f+g)(x_*)\big]\le\limsup_{k\to\infty}\, k\big[(f+g)(x^{k})-(f+g)(x_*)\big]\le 0,
\]
thanks to the fact that  $x_*\in S_*$. Hence we obtain 
$
\lim_{k\to\infty} k\big[(f+g)(x^{k})-(f+g)(x_*)\big]= 0,
$
which verifies \eqref{xrate} and completes the proof of theorem. 
\end{proof}
{It is worth mentioning that the rate $o(k^{-1})$ was obtained  \cite{Dong2, Dong3, guller} earlier when using the proximal point method to solve problem \eqref{prob} with $f\equiv 0$. \footnote{{This important observation is pointed out from by one of the referees}} Our result above could be considered  an extension of  some results in these papers, in particular, \cite[Corollary~3.1]{Dong2} to the more general framework of \eqref{prob} with linesearch.   When the the stepsizes are not bounded below by a positive constant, we discuss the possible validity of the same complexity as follows.}
%%%%%%%%%%%%%%%%%%%%%%%%%%%%%%%%%%%%%
\begin{remark}\label{remark-4.1}{\rm  The main question arising from the above theorem is that: Can we have the complexity $o(k^{-1})$ of the difference $(f+g)(x^k)-\min_{x\in\HH}(f+g)(x)$ when $\liminf_{k\to \infty}\al_k=0$? Suppose that $(x^k)_{k\in\NN}$ (strongly) converges to some $x^*\in S_*$ in finite dimension; see our Theorem~\ref{new-cov}. By analyzing carefully the proof of \eqref{rate} in Theorem \ref{l-rate}, we observe that  complexity $o(k^{-1})$ remains when the following condition holds: there exists $\lambda\in [-1,1)$ such that 
\begin{equation}\label{cond-weak-alpha}
\limsup_{k\to\infty}\frac{\|x^k-x_*\|^{1+\lambda}}{\al_k}<+\infty, 
\end{equation}
which may allow $\al_k$ to approach $0$. Indeed, suppose that \eqref{cond-weak-alpha} is satisfied with some $\lambda\in [-1,1)$, we find $C>0$ and $K\in \NN$ such that $\|x^k-x_*\|^{1+\lambda}\le C{\al_k}$ for all $k>K$. For any $\ve\in (0,1)$, there exists $K_1>K$ such that $\|x^k-x_*\|<\ve$ for all $k\ge K_1$. Moreover, it is easy to prove the existence of some constant $D\ge1$ so that 
\begin{equation}
\label{new-eq-rho}
(\rho^2-1)\le 2D\rho^{1+\lambda}(\rho^{1-\lambda}-1)\quad \mbox{for all}\quad \rho\ge 1. 
\end{equation}
Note again that $\disp\frac{\|x^k-x_*\|}{\|x^{k+1}-x_*\|}\ge 1$ due to the Fej\'er property of $(x^k)_{k\in \NN}$ in Theorem~\ref{new-cov}{(i)}. This together with \eqref{new-eq-rho} tells us that
\[
\|x^k-x^*\|^2-\|x^{k+1}-x^*\|^2\le 2D \|x^k-x^*\|^{1+\lambda}(\|x^k-x^*\|^{1-\lambda}- \|x^{k+1}-x^*\|^{1-\lambda}).
\]
Hence for any $\ell>K_1$ we get from \eqref{antes-sum} that   
 \begin{eqnarray*}\begin{array}{ll}
0\ge(f+g)(x_*)-(f+g)(x^{\ell+1})&\disp\ge D\,\frac{\|x^\ell-x_*\|^{1+\lambda}}{\alpha_\ell}\big(
\|x^{\ell+1}-x_*\|^{1-\lambda}-\|x^\ell-x_*\|^{1-\lambda}\big)\\
&\disp\ge CD\big(
\|x^{\ell+1}-x_*\|^{1-\lambda}-\|x^\ell-x_*\|^{1-\lambda}\big). 
\end{array}
\end{eqnarray*}
By adding the above
inequality over $\ell=K_1,K_1+1,\ldots,K_1+k-1$, we have 
\begin{equation*}
 k(f+g)(x_*)-\sum_{\ell=K_1}^{K_1+k-1}(f+g)(x^{\ell+1})\ge\disp CD\big(
\|x^{k+K_1}-x_*\|^{1-\lambda}-\|x^{K_1}-x_*\|^{1-\lambda}\big).
\end{equation*}
Due to the decreasing property of  $(f+g)(x^{\ell})$ in Proposition~\ref{prop1}(ii),  the latter implies that 
\begin{equation*}
 k\big[(f+g)(x_*)-(f+g)(x^{K_1+k})\big]\ge -CD\,\|x^{K_1}-x_*\|^{1-\lambda}\ge-CD\ve^{1-\lambda}.
\end{equation*}
Thus we derive the following expressions 
\begin{align*}
\limsup_{k\to\infty}\, k\big[(f+g)(x^{k})-(f+g)(x_*)\big]&=\limsup_{k\to\infty}\, (K_1+k)\big[(f+g)(x^{K_1+k})-(f+g)(x_*)\big]\\
&\le\limsup_{k\to\infty}\,\frac{K_1+k}{k}CD\ve^{1-\lambda}=CD\ve^{1-\lambda}. 
\end{align*}
Since this inequality holds for any $\ve>0$, we have 
$
\limsup_{k\to\infty}\, k\big[(f+g)(x^{k})-(f+g)(x_*)\big]\le 0,
$
which also verifies \eqref{xrate} due to the fact that $x_*\in S_*$. \lqqd}
\end{remark}\vspace*{-0.2in}
It is clear that \eqref{cond-weak-alpha} holds when $\al_k$ is bounded below by a positive number. The following simple example shows the possible validity of \eqref{cond-weak-alpha} even when $\al_k\to 0$ as $k\to\infty$. Thus the complexity $o(k^{-1})$ of the function values remains true in the example below. However, in general, checking  \eqref{cond-weak-alpha} may be not trivial, since $x_*$ is  unknown.
\begin{example}
{\rm Let 
\[
f(x):=\frac{1}{1+p}\,|x|^{1+p}\quad \mbox{with}\; 0<p<1\quad \mbox{and}\quad g(x)=\delta_{[0,\infty)}(x).
\] 
Then a unique solution for problem \eqref{prob} is $x_*=0$. Note further that for any $x>0$, we have  
\begin{eqnarray}\label{xy}
J(x,\alpha)=P_{[0,+\infty)}(x-\alpha x^p)=\max\{x-\alpha x^p,0\}.
\end{eqnarray}
To distinguish the iteration from the exponent in this example, we write $(x_k)_{k\in \NN}$ instead of $(x^k)_{k\in \NN}$. To avoid the trivial case, suppose that $x_k> 0$ for all $k\in \NN$, then we have $$0<x_{k+1}= x_k-\al_k (x_k)^p<x_k.$$ It follows from {\bf Linesearch~1} that 
\begin{equation}\label{ab}
\al_k|x_{k+1}^p-x_k^p|\le \delta|x_{k+1}-x_k|.
\end{equation}
By mean value theorem, there exists $\eta\in[0,1]$ such that 
\begin{equation*}
|x_{k+1}^p-x_k^p|=|x_{k+1}-x_k|\cdot p|\eta\, x_{k+1}+(1-\eta)x_k|^{p-1} \ge  |x_{k+1}-x_k|\cdot p|x_k|^{p-1}.
\end{equation*}
This together with \eqref{ab} gives us  $\al_k\le \delta p^{-1}|x_k|^{1-p}\to 0$ as $k\to \infty$, since $x_k\to 0$. Therefore, we may suppose without loss of generality that $\al_k<\sigma$ for all $k$. Define $\displaystyle\hat \al_k:=\frac{\al_k}{\theta}$ and $\hat x_{k+1}=J(x_k,\hat \al_k)$, it follows from the {\bf Linesearch~1} that 
\begin{equation}\label{abb}
\hat\al_k|\hat x_{k+1}^p-x_k^p|> \delta|\hat x_{k+1}-x_k|.
\end{equation}
Note that $0\le \hat x_{k+1}<x_k$ by \eqref{xy} and that 
$$
0\le x_k^p-\hat x_{k+1}^p\le x_k^p-x_{k}^{p-1}{\hat x_{k+1}}=x_{k}^{p-1}(x_k-{\hat x_{k+1}}).
$$
Combining this with \eqref{abb} gives us that 
$\displaystyle
\frac{\al_k}{\theta} |x_{k}|^{(p-1)}|x_k-{\hat x_{k+1}}|\ge \delta|\hat x_{k+1}-x_k|>0,
$
which implies that $\displaystyle\frac{|x_k-x_*|^{1-p}}{\al_k}\le \frac{1}{\theta}$. This is exactly \eqref{cond-weak-alpha} with $\lambda=-p\in [-1,1)$.\lqqd}
\end{example}\vspace*{-0.2in}
Another natural question from Theorem~\ref{l-rate} is that in which class of functions the stepsizes $\al_k$ are bounded below by a positive number. Next we show that this condition is satisfied under some mild  Lipschitz continuity assumption of $\nabla f$. The first part of this result is not much surprising due to the similar achievement in \cite[Theorem~3.4(a)]{tseng}.
However, the second part is a significant improvement when we replace the global Lipschitz continuity by the local one in finite dimensions.
\begin{proposition}\label{lema-alpha}
Let $(\alpha_k)_{k\in \NN}$ be the sequence generated by {\bf
Linesearch~\ref{boundary}} on {\bf Method \ref{A1}}. The following statements hold:
\item {\bf (i)} If the gradient
of $f$ is globally Lipschitz continuous on $\dom g$ with constant $L>0$, then
$\alpha_k\ge \min\left\{\sigma, \frac{\delta\theta}{L}\right\}$ for
all $k\in \NN$.

\item {\bf (ii)} Suppose that $\dim \HH<+\infty$ and $S_*\neq\emptyset$. If $\nabla f$ is locally Lipschitz continuous
at any $x\in S_*$ then there exists  $x_*\in S_*$ such that
$$
\liminf_{k\to\infty}\al_k\ge
\min\Big\{\sigma,\frac{\delta\theta}{\mathcal{L}}\Big\},
$$
where $\mathcal{L}>0$ is a Lipschitz constant of $\nabla f$
around $x_*$. Consequently, there exists $\al>0$ such that $\al_k\ge
\al$ for all $k\in \NN$.
\end{proposition}
\begin{proof}
To justify {\bf (i)}, suppose that $\nabla f$ is globally Lipschitz continuous with constant $L>0$.  If $\alpha_k<\sigma$, define $\hat{\alpha}_k:=\dsty
\frac{\alpha_k}{\theta}>0$ and $
\hat{x}^k:=J(x^k,\hat{\alpha}_k)$. It follows from  the
definition of {\bf Linesearch \ref{boundary}}  that
\begin{equation}\label{eq-no-a}
\hat{\alpha}_k\left\|\nabla f\big(\hat{x}^k\big)-\nabla
f(x^k)\right\|>\delta\left\|\hat{x}^k-x^k\right\|,
\end{equation}
which yields $\|\hat{x}^k-x^k\|\neq0$ for all $k\in\NN$. Moreover,  due to Lipschitz assumption on $\nabla f$, we get $ \|\nabla f(x^{k})-\nabla
f(\hat{x}^{k})\|\le L \|x^{k}-\hat{x}^{k}\|$ for all $k\in \NN. $
Combining the latter inequality with \eqref{eq-no-a} gives us that
$\hat \alpha_k L> \delta$, \emph{i.e.}, $ \alpha_k\ge \frac{\delta\theta}{L}
$ when $\alpha_k<\sigma$. This clearly verifies {\bf (i)}.

To justify the second part, we suppose that $\dim \HH<+\infty$, that
$S_*\neq\emptyset$, and that $f$ is locally Lipschitz continuous at
any point in $S_*$. By Theorem~\ref{new-cov}, $(x^k)_{k\in \NN}$
converges (strongly) to some $x_*\in S_*$. Due to the local Lipschitz continuity of $\nabla f$ at $x_*$, there exist $\ve, \mathcal{L}>0$ such that 
\begin{equation}\label{sL}
\|\nabla f(x)-\nabla f(y)\|\le \mathcal{L}\|x-y\|\quad \mbox{for all} \quad x,y\in \mathbb{B}_\ve(x_*),
\end{equation}
where $\mathbb{B}_\ve(x_*)$ is the closed ball in $\HH$ with center  $x_*$ and radius $\ve$. Since $(x^k)_{k\in \NN}$ is converging (strongly) to $x_*$, we find some $K\in \NN$ satisfying that
\begin{equation}\label{ssL}
\|x^k-x_*\|\le \frac{\theta\ve}{2+\theta}<\ve\quad \mbox{for all}\quad k>K
\end{equation}
with $\theta\in(0,1)$ defined in {\bf
Linesearch~\ref{boundary}}. Take any $k>K$, if $\al_k<\sigma$, similarly to the first part we  define
$\hat{\alpha}_k:=\frac{\alpha_k}{\theta}>0$ and $
\hat{x}^k:=J(x^k,\hat{\alpha}_k)$. Thus we also have
\eqref{eq-no-a}. It follows from
Lemma~\ref{lema-prop} that
\[
\|x^{k}-\hat x^{k}\|=\|x^{k}-J(x^{k},\hat \alpha_{k})\|\le \frac{\hat \al_{k}}{\al_{k}}\|x^{k}-J(x^{k},\al_{k})\|= \frac{1}{\theta}\|x^{k}-x^{k+1}\|,
\]
which together with \eqref{ssL} implies the following expression
\[
\|\hat x^k-x_*\|\le \| \hat x^k- x^k\|+\| x^k-x_*\|\le \frac{1}{\theta}\|x^{k}-x^{k+1}\|+\| x^k-x_*\|\le \frac{1}{\theta}\cdot \frac{2\theta\ve}{2+\theta}+\frac{\theta\ve}{2+\theta}=\ve. 
\]
Hence we have $\hat x^k\in \mathbb{B}_\ve(x_*)$ and derive from \eqref{sL} and \eqref{ssL} that 
$
\|\nabla f(x^k)-\nabla f(\hat x^k)\|\le \mathcal{L}\|x^k-\hat x^k\|.
$
Combining this with \eqref{eq-no-a} gives us that $\mathcal{L}\hat
\al_k\ge\delta$, \emph{i.e.}, $\al_k\ge\frac{\delta\theta}{\mathcal{L}}$. It follows that $\al_k\ge \min\big\{\sigma,\frac{\delta\theta}{\mathcal{L}}\big\}$ for all $k>K$.

\noindent Finally, since $\al_k>0$ for $k\in \NN$, we obtain that
$
\al_k\ge
\al:=\min\left\{\al_1,\ldots,\al_{K},\frac{\delta\theta}{\mathcal{L}},\sigma\right\}>0
$
 and ensure the last part of the proposition. The proof is complete. 
\end{proof}
 It is worth recalling that the assumption of Proposition~\ref{lema-alpha}(i) that $\nabla f$ is globally Lipschitz continuous on $\dom g$
is also sufficient for  Assumption {\bf A2}. Assumptions of
Proposition~\ref{lema-alpha}(ii) are certainly not enough to
guarantee Assumption {\bf A2}. However, there are many broad classes of
functions satisfying all of them. For instance, when $\dim
\HH<+\infty$ and $\dom g$ is closed, a function $f$, which is
differentiable with locally Lipschitz continuous  gradient on
$\dom g$  satisfies all the requirements; see also
Proposition~\ref{pop}.

Theorem~\ref{l-rate} together with Proposition~\ref{lema-alpha} and Theorems~\ref{with-strong} leads us to the following result. Unlike \cite[Theorem~1.1]{beck-teu}, we obtain  better complexity $o(k^{-1})$ with linesearches in finite dimensions for a broader class of functions.    
\begin{corollary}\label{with-Lip}
Let $(x^k)_{k\in \NN}$ be the sequence generated by
{\bf Method~\ref{A1}}. Suppose that $S_*\neq \emptyset$. 

\item {\bf (i)} If the gradient of $f$ is globally Lipschitz continuous  on $\dom g$, then we have 
\[
(f+g)(x^k)-\min_{x\in \HH}\,(f+g)(x)=\mathcal{O}(k^{-1}).
\]
\item {\bf(ii)} If $\dim \HH<+\infty$ and the gradient of $f$ is locally  Lipschitz continuous  on
$S_*$,  then we have
\[
(f+g)(x^k)-\min_{x\in \HH}\,(f+g)(x)=o(k^{-1}).
\]
\end{corollary}
We obtain linear convergence when the stepsizes are bounded below
by a positive number and either $f$ or $g$ is strongly convex. Recall
that $h:\HH\to\overline{\RR}$ is strongly convex with constant
$\mu>0$ if,
$$h(x)\ge h(y)+\la
v,x-y\ra+\frac{\mu}{2}\|x-y\|^2\quad \mbox{for all}\quad x\in \HH,\, (y,v)\in {\rm Gph}\, \partial h.$$
\begin{theorem}\label{with-strong}
Let $(x^k)_{k\in \NN}$ and $(\al_k)_{k\in \NN}$ be the sequences
generated in {\bf Method~\ref{A1}}. Suppose that $S_*\neq\emptyset$,
that there exists $\al>0$ satisfying $\alpha_k\ge\alpha>0$
for all $k\in \NN$, and that either $f$ or $g$ is strongly convex
with constant $\mu>0$. Then $S_*=\{x_*\}$ is singleton and
\begin{equation}\label{linear-s}
\|x^{k+1}-x_*\|\le \frac{1}{\sqrt{1+\alpha\mu}}\cdot\|x^k-x_*\|\le
\left(\frac{1}{\sqrt{1+\alpha\mu}}\right)^{k+1}\|x^0-x_*\|\quad
\forall\, k\in \NN,
\end{equation}
i.e., the sequence $(x^k)_{k\in \NN}$ converges (strongly) to $x_*$ with the
linear rate $\dsty\frac{1}{\sqrt{1+\alpha\mu}}<1$.

Consequently,  if either $f$ or $g$ is strongly convex,  $\nabla f$ is locally Lipschitz continuous on $S_*$, and $\dim \HH<+\infty$, then $(x^k)_{k\in \NN}$ converges linearly to the unique optimal solution.
\end{theorem}
\begin{proof}
Since  either $f$ or $g$ is strongly convex with constant $\mu>0$,  $f+g$ is
also strongly convex with constant $\mu>0$. It follows that $S_*$ is singleton
(\emph{i.e.}, $S_*=\{x_*\}$). Moreover, using Proposition
\ref{prop1}(i) with $x=x_*$ and the strong convexity of $f+g$ gives
us that
\begin{align*}
\|x^{k}-x_*\|^2\ge&
\|x^{k+1}-x_*\|^2+2\alpha_k[(f+g)(x^{k+1})-(f+g)(x_*)]\\\nonumber\ge&
\|x^{k+1}-x_*\|^2 +
\alpha_k\mu\|x^{k+1}-x_*\|^2\ge(1+\alpha\mu)\|x^{k+1}-x_*\|^2.
\end{align*}
It follows that
$$
\|x^{k+1}-x_*\|\le \frac{1}{\sqrt{1+\alpha\mu}}\cdot\|x^{k}-x_*\|\le
\left(\frac{1}{\sqrt{1+\alpha\mu}}\right)^{k+1}\|x^0-x_*\|,
$$
which verifies \eqref{linear-s} and thus completes the proof of the theorem.
\end{proof}

Since the condition $x=J(x,\alpha)$ for $\alpha>0$ is necessary and
sufficient for $x$ to be an optimal solution to problem
\eqref{prob}, it is interesting to study the complexity of
$\|x^k-J(x^k,\al_k)\|$ in our {\bf Method~\ref{A1}}.  The velocity
of the convergence obtained below is not affected by the behavior of
the stepsizes $\alpha_k$.
\begin{theorem}\label{no-convexo}
Let $(x^k)_{k\in \NN}$ and $(\al_k)_{k\in \NN}$ be the sequences
generated from {\bf Method~\ref{A1}}. Then we have
\begin{equation}\label{com}
\liminf_{k\rightarrow\infty}
\sqrt{k}\,\cdot\|x^k-J(x^k,\alpha_k)\|=0.
\end{equation}
\end{theorem}
\begin{proof}
If \eqref{com} does not hold, then we may find a number $\ve>0$ such
that for some fixed $K\in\NN$ large enough, we have
$\dsty\|x^k-J(x^k,\alpha_k)\| \ge \frac{\varepsilon}{\sqrt{k}}$ for
all $k\ge K$. Thus,
\begin{equation}\label{ultima-eq}
\sum_{k=K}^\infty \|x^k-J(x^k,\alpha_k)\|^2 \ge\varepsilon^2
\sum_{k=K}^\infty\frac{1}{k}=+\infty. \end{equation} On the
other hand, using \eqref{F-B} and Proposition \ref{prop1}(ii), we
get, for all $k\ge K$,
\begin{align*}
\|x^k-J(x^k,\alpha_k)\|^2&=\|x^k-x^{k+1}\|^2\le
\frac{\alpha_k}{1-\delta}\left[(f+g)(x^k)-(f+g)(x^{k+1})\right]\\&\le\frac{\sigma}{1-\delta}\left[(f+g)(x^k)-(f+g)(x^{k+1})\right],
\end{align*}
where we have used in the last inequality that
$\alpha_k\le\sigma$ for all $k\in\NN$, which follows from  {\bf Linesearch \ref{boundary}}. Hence, we have
$$
\sum_{k=K}^\infty
\|x^k-J(x^k,\alpha_k)\|^2\le\frac{\sigma}{1-\delta}\left[(f+g)(x^{K})-(f+g)(x_*)\right]<+\infty,
$$ which contradicts \eqref{ultima-eq}. The proof is complete.
\end{proof}

%%%%%%%%%%%%%%%%%%%%%%%%%%%%%%%%%%%%%%%%%%%%%%%%%%%%%%%%%%%%%%%%%%%%%%%%%%%%%%%%%%%%

\subsection{A fast multistep forward-backward method with Linesearch \ref{boundary}}\label{s:4}
In the spirit of the classical work of Nesterov \cite{nesterov-1983}
many accelerated multistep versions have been proposed in the
literature for the forward-backward iteration, but to the best of our knowledge all of them have to employ the global Lipschitz continuity assumption on $\nabla f$; see,
\emph{e.g.},\cite{beck-teu, beck, nesterov-2013}. In this subsection, by following these ideas and assuming no Lipschitz continuity on
$\nabla f$, we present a fast version of the proximal
forward-backward method with {\bf Linesearch~\ref{boundary}},
improving the convergence result of Theorem \ref{l-rate} for {\bf
Method~\ref{A1}}. In \cite{beck-teu, beck, nesterov-2013} this kind of fast versions
usually demands Lipschitz assumption over $\nabla f$  to establish convergence of this method. Here we modify the method by adding a linesearch and an extra projection step in
\eqref{yk} below to avoid the requirements aforementioned. For simplicity, we suppose   $\Omega:=\dom g$ is closed in this section.
\begin{center}\fbox{\begin{minipage}[b]{\textwidth}
\begin{method}\label{A3}
\item [    ] {\bf Initialization Step.} Take $x^{-1}=x^0\in \dom g$, $t_0=1$, $\theta\in(0,1)$, $\alpha_{-1}=\sigma$  and $\delta\in(0,1/2)$.

\item [    ] \noindent {\bf Iterative Step.} Given $t_k$ and $x^k$, set
\begin{eqnarray}\label{tk+1}
t_{k+1}&=&\frac{1+\sqrt{1+4t_k^2}}{2} \\
y^{k}&=&x^{k}+\left(\frac{t_{k}-1}{t_{k+1}}\right)(x^{k}-x^{k-1}),
\quad \tilde{y}^k=P_\Omega(y^k)\label{yk}\\
x^{k+1}&=&J(\tilde{y}^k,\alpha_k):=\prox_{\alpha_k
g}(\tilde{y}^k-\alpha_k \nabla f(\tilde{y}^k)) \label{F-B-2}
\end{eqnarray}
with $\alpha_k:=$ {\bf Linesearch
\ref{boundary}}$(\tilde{y}^k,\alpha_{k-1},\theta,\delta)$.

\item [    ] \noindent  {\bf Stop Criteria.} If $x^{k+1}=\tilde{y}^k$, then stop.
\end{method}\end{minipage}}\end{center}
Note that from \eqref{yk} and \eqref{F-B-2}, $\tilde{y}^k$ and $x^k$
belong to $\dom g$ for all $k\in\NN$ and as a direct consequence of
Lemma~\ref{boundary-well}, $\alpha_k$ satisfying \eqref{conda-2} is
always positive and nonincreasing. Moreover, it is
similar to {\bf Method~\ref{A1}} that if $x^{k+1}=\tilde y^k$ then
$x^{k+1}$ is an optimal solution.
%To avoid finite termination of this method, we suppose from now onthat $x^{k+1}\neq \tilde y^k$ for any $k\in \NN$.
An important
inequality for our further study from {\bf
Linesearch~\ref{boundary}} is
\begin{equation}\label{conda-2}
\alpha_k\left\|\nabla f\big(x^{k+1}\big)-\nabla
f(\tilde{y}^k)\right\|\leq\delta\left\|x^{k+1}-\tilde{y}^k\right\|
\end{equation}
with  $\delta\in(0,1/2)$.   We also need some auxiliary results before establishing the convergence results.
\begin{lemma}\label{aux-1} The positive sequence $(t_k)_{k\in \NN}$ generated by {\bf Method~\ref{A3}} via \eqref{tk+1} satisfies,  for all $k\in
\NN$, \item {\bf (i)} $\dsty\frac{1}{t_k}\le \frac{2}{k+1}$;
\item {\bf (ii)} $t^2_{k+1}-t_{k+1}=t^2_k$.
\end{lemma}
\begin{proof}
The proof easily follows by induction argument.
\end{proof}
\begin{proposition}\label{prop1-2} Let $\alpha_k$ be defined in {\bf Method~\ref{A3}}  and  $x\in \dom g$.
Then we have
\begin{equation}\label{nn}
(f+g)(x)-(f+g)(x^{k+1})\ge \frac{1}{2\alpha_k}\left(\|x^{k+1}-x\|^2-\|y^k-x\|^2\right)\quad \mbox{for all}\quad  k\in \NN.
\end{equation}
\end{proposition}
\begin{proof}
First note from \eqref{in-sub} with $z=\tilde{y}^k-\alpha_k \nabla
f(\tilde{y}^k)$ that $\dsty\frac{\tilde{y}^k-x^{k+1}}{\alpha_k}-\nabla f(\tilde{y}^k)\in\partial g(x^{k+1})$. Then,
\begin{equation}\label{eq2-2}g(x)-g(x^{k+1})\ge \left\la\frac{\tilde{y}^k-x^{k+1}}{\alpha_k}-\nabla f(\tilde{y}^k),x-x^{k+1}\right\ra
\end{equation}
for all $x\in \dom g$.
The convexity of $f$ implies that
\begin{equation}\label{eq3-2}
f(x)-f(y)\ge \la\nabla f(y), x-y\ra\quad \mbox{for all}\quad x\in \dom f\mbox{ and } y\in \dom g.
\end{equation}
 By summing \eqref{eq2-2}  and \eqref{eq3-2} with $y=\tilde{y}^k\in \Omega=\dom g$, we
obtain that
\begin{align*}(f+g)(x)\ge& f(\tilde{y}^k)+g(x^{k+1}) + \left\la\frac{\tilde{y}^k-x^{k+1}}{\alpha_k}-\nabla f(\tilde{y}^k),x-x^{k+1}\right\ra
+\la\nabla f(\tilde{y}^k), x-\tilde{y}^k\ra\\ =&
f(\tilde{y}^k)+g(x^{k+1}) +\frac{1}{\alpha_k} \left\la
\tilde{y}^k-x^{k+1},x-x^{k+1}\right\ra+\la\nabla f(\tilde{y}^k), x^{k+1}-\tilde{y}^k\ra\\
=&f(\tilde{y}^k)+g(x^{k+1}) +\frac{1}{\alpha_k} \left\la
\tilde{y}^k-x^{k+1},x-x^{k+1}\right\ra+\la\nabla
f(\tilde{y}^k)-\nabla f(x^{k+1}),
x^{k+1}-\tilde{y}^k\ra\\& +\la\nabla f(x^{k+1}),x^{k+1}-\tilde{y}^k\ra\\
\ge &f(\tilde{y}^k)+g(x^{k+1}) +\frac{1}{\alpha_k} \left\la
\tilde{y}^k-x^{k+1},x-x^{k+1}\right\ra- \frac{\delta}{\alpha_k}\|
x^{k+1}-\tilde{y}^k\|^2\\&+\la\nabla
f(x^{k+1}),x^{k+1}-\tilde{y}^k\ra,
\end{align*}
where the last inequality follows from \eqref{conda-2}.
Rearranging the inequality gives us that
\begin{align}\nonumber\label{eq5-2}
\la \tilde{y}^k-x^{k+1},x^{k+1}-x \ra\ge&\disp
\alpha_k[f(\tilde{y}^k)+g(x^{k+1})-(f+g)(x)]-\delta\|
x^{k+1}-\tilde{y}^k\|^2\\&\disp+ \alpha_k\la\nabla
f(x^{k+1}),x^{k+1}-\tilde{y}^k\ra.
\end{align}
Observe that
$
2\la \tilde{y}^k-x^{k+1},x^{k+1}-x
\ra=\|\tilde{y}^k-x\|^2-\|x^{k+1}-x\|^2-\|\tilde{y}^k-x^{k+1}\|^2.
$
By combining the above equality with \eqref{eq5-2}, we have
\begin{align}\nonumber\label{eq6-2}\|\tilde{y}^k-x\|^2-\|x^{k+1}-x\|^2\ge&2\alpha_k\left [f(\tilde{y}^k)+g(x^{k+1})-(f+g)(x)+
\la\nabla f(x^{k+1}),x^{k+1}-\tilde{y}^k\ra\right]\\
\nonumber &+ (1-2\delta)\|\tilde{y}^k-x^{k+1}\|^2\\
\ge&2\alpha_k\left[f(\tilde{y}^k)+g(x^{k+1})-(f+g)(x)+\la\nabla
f(x^{k+1}),x^{k+1}-\tilde{y}^k\ra\right].\end{align}
It follows from \eqref{eq3-2} with $x=\tilde{y}^k$ and $y=x^{k+1}$ that
$f(\tilde{y}^k)-f(x^{k+1})\ge\la\nabla f(x^{k+1}),\tilde{y}^k-x^{k+1}\ra$, which together with \eqref{eq6-2} implies
\begin{align*}
\|\tilde{y}^k-x\|^2-\|x^{k+1}-x\|^2\ge&2\alpha_k\left
[f(\tilde{y}^k)+g(x^{k+1})-(f+g)(x)+f(x^{k+1})-f(\tilde{y}^k)\right]\\
=&2\alpha_k\left[(f+g)(x^{k+1})-(f+g)(x)\right].
\end{align*}
Since $\|\tilde{y}^k-x\|\le\|{y}^k-x\|$ for all $x\in \dom g$ due to \eqref{yk}, we get from the latter \eqref{nn} and complete the proof of the proposition.
\end{proof}
In the next result we establish a better complexity
for {\bf Method~\ref{A3}} than {\bf Method~\ref{A1}} in  Theorem \ref{l-rate} under a similar assumption.
\begin{theorem}\label{l-rate-2}
Let $(x^k)_{k\in \NN}$ and $(\al_k)_{k\in \NN}$ be the sequences
generated in  {\bf Method~\ref{A3}}. Suppose that $S_*\neq\emptyset$ and there is $\al>0$ such that  $\alpha_k\ge\alpha>0$ for
all $k\in \NN$. Then we have
\begin{equation*}\label{rate-2}
(f+g)(x^k)-\min_{x\in \HH}\,(f+g)(x)\le
\frac{\dsty\frac{2}{\alpha}\cdot\left(
\|x^0-x_*\|^2+2\sigma\left[(f+g)(x^{0})-\dsty\min_{x\in
\HH}(f+g)(x)\right]\right)}{(k+1)^2} \quad \mbox{for all}\;\; k\in
\NN.
\end{equation*}
\end{theorem}
\begin{proof} To justify, pick any $x_*\in S_*$.
By Lemma \ref{aux-1}(i)  and the convexity of $g$, we have $t_{k+1}\ge 1$ and thus
$x:=t^{-1}_{k+1}x_*+(1-t^{-1}_{k+1})x^{k}\in \dom g$. Applying Proposition~\ref{prop1-2} for this $x$ gives us that
\begin{align*}
&\frac{1}{2\alpha_k}\left(\left\|x^{k+1}-\left(t^{-1}_{k+1}x_*+\left(1-t^{-1}_{k+1}\right)x^{k}\right)\right\|^2-
\left\|y^k-\left(t^{-1}_{k+1}x_*+\left(1-t^{-1}_{k+1}\right)x^{k}\right)\right\|^2\right)\\
&\le
(f+g)(t^{-1}_{k+1}x_*+\left(1-t^{-1}_{k+1}\right)x^{k})-(f+g)(x^{k+1})\\
&\le
t^{-1}_{k+1}(f+g)(x_*)+(1-t^{-1}_{k+1})(f+g)(x^{k})-(f+g)(x^{k+1}).
\end{align*}
After rearrangement, we obtain
\begin{align*}\nonumber
&(1-t^{-1}_{k+1})\left[(f+g)(x^{k})-(f+g)(x_*)\right]-\left[(f+g)(x^{k+1})-(f+g)(x_*)\right]\\
&\ge
\frac{1}{2\alpha_kt^{2}_{k+1}}\left(\left\|t_{k+1}x^{k+1}-(x_*+(t_{k+1}-1)x^k)\right\|^2-\left\|t_{k+1}y^k-(x_*+(t_{k+1}-1)x^{k})\right\|^2\right).
\end{align*}
By multiplying by $t^2_{k+1}$ to the above inequality and  using \eqref{yk} and Lemma
\ref{aux-1}(ii), we have
\begin{align*}\label{antes-s-3}\nonumber
&\frac{1}{2\alpha_k}\left(\|t_{k+1}x^{k+1}-(x_*+(t_{k+1}-1)x^k)\|^2-\|t_{k+1}y^k-(x_*+(t_{k+1}-1)x^{k})\|^2\right)\\
&=\frac{1}{2\alpha_k}\left(\|t_{k+1}x^{k+1}-(t_{k+1}-1)x^k-x_*\|^2-\|t_{k}x^k-(t_{k}-1)x^{k-1}-x_*\|^2\right)\\
&\le(t^2_{k+1}-t_{k+1})\left[(f+g)(x^{k})-(f+g)(x_*)\right]-t^2_{k+1}\left[(f+g)(x^{k+1})-(f+g)(x_*)\right]\\
&=t^2_{k}\left[(f+g)(x^{k})-(f+g)(x_*)\right]-t^2_{k+1}\left[(f+g)(x^{k+1})-(f+g)(x_*)\right].
\end{align*}
It follows that
\begin{align}\nonumber
&\|t_{k}x^k-(t_{k}-1)x^{k-1}-x_*\|^2-\|t_{k+1}x^{k+1}-(t_{k+1}-1)x^k-x_*\|^2\\
\nonumber&\ge 2\alpha_k
\left(t^2_{k+1}\left[(f+g)(x^{k+1})-(f+g)(x_*)\right]-
t^2_{k}\left[(f+g)(x^{k})-(f+g)(x_*)\right]\right)\\ \nonumber&\ge
2\alpha_{k+1}
t^2_{k+1}\left[(f+g)(x^{k+1})-(f+g)(x_*)\right]-
2\alpha_k t^2_{k}\left[(f+g)(x^{k})-(f+g)(x_*)\right],
\end{align}
where the last inequality follows from the facts that $\al_k\ge \al_{k+1}={\bf Linesearch~1}(\tilde y^k,\al_k,\theta,\delta)$ and $(f+g)(x^{k+1})-(f+g)(x_*)\ge 0$.
Reordering the above inequality and applying it inductively yield
\begin{align*}\nonumber
&2\alpha_{k+1}
t^2_{k+1}\left[(f+g)(x^{k+1})-(f+g)(x_*)\right]\\ \nonumber &
\le\|t_{k+1}x^{k+1}-(t_{k+1}-1)x^k-x_*\|^2+2\alpha_{k+1}
t^2_{k+1}\left[(f+g)(x^{k+1})-(f+g)(x_*)\right]\\ \nonumber&\le
\|t_{k}x^k-(t_{k}-1)x^{k-1}-x_*\|^2+2\alpha_{k}
t^2_{k}\left[(f+g)(x^{k})-(f+g)(x_*)\right]\\ \nonumber &\le \ldots\le \|t_{0}x^0-(t_{0}-1)x^{-1}-x_*\|^2+2\alpha_{0}
t^2_{0}\left[(f+g)(x^{0})-(f+g)(x_*)\right]\\ \nonumber&=
\|x^0-x_*\|^2+2\alpha_{0}\left[(f+g)(x^{0})-(f+g)(x_*)\right],
\end{align*}
which readily imply
$
2\alpha_{k}
t^2_{k}[(f+g)(x^{k})-(f+g)(x_*)]\le\|x^0-x_*\|^2+2\sigma\left[(f+g)(x^{0})-(f+g)(x_*)\right].
$
Using this inequality together with Lemma \ref{aux-1}(i) gives us that
\begin{align*}(f+g)(x^{k})-\min_{x\in \HH}\,(f+g)(x)\le & \frac{1}{2\alpha_{k} t^2_k}
\left(\|x^0-x_*\|^2+2\sigma\left[(f+g)(x^{0})-\min_{x\in \HH}\,(f+g)(x)\right]\right)\\
\le & \frac{\dsty\frac{2}{\alpha}\cdot\left(
\|x^0-x_*\|^2+2\sigma\left[(f+g)(x^{0})-\dsty\min_{x\in
\HH}(f+g)(x)\right]\right)}{(k+1)^2}
\end{align*} for all $x_*\in
S_*$ and thus verifies \eqref{nn}. The proof of the theorem is
complete.
\end{proof}
This theorem shows that the expected error of the iterates
generated by {\bf Method~\ref{A3}} after $k$ iterations is
$\mathcal{O}(k^{-2})$ when the stepsizes are bounded below by a positive constant. Similarly to Proposition \ref{lema-alpha}, we prove in the next result that  such a requirement is satisfied under global Lipschitz assumption on the gradient of $f$. {The complexity $o(k^{-2})$ for the accelerated scheme similarly to \eqref{tk+1}--(45) has been obtained recently in \cite{chambolle, attouch}  under the global Lipschitz assumption. It would be interesting to combine their techniques with ours to derive similar complexity under the weaker assumption of local Lipschitz continuity as in Proposition~\ref{lema-alpha}(ii).}  
\begin{proposition}\label{ac-lema-alpha}
Let $(\alpha_k)_{k\in \NN}$ be the sequence generated by {\bf
Linesearch~\ref{boundary}} on {\bf Method \ref{A3}}. If the gradient
of $f$ is globally Lipschitz continuous on $\dom g$  then there exists some $\al>0$ such that
$\alpha_k\ge \al$ for
all $k\in \NN$.
\end{proposition}
\begin{proof} Suppose that $\nabla f$ is globally Lipschitz continuous on $\dom g$ with constant $L>0$. Since $\al_k$ is nonnegative and decreasing,  $\lim_{k\to\infty}\al_k=\al$ exists.  If $\al<\frac{\delta\theta}{L}$, we may find   $K\in \NN$ such that $\al_k<\frac{\delta\theta}{L}$ for all $k>K$.
Define further $\hat{\alpha}_k:=\dsty \frac{\alpha_k}{\theta}>0$, and $
\hat{y}^k:=J(\tilde{y}^k,\hat{\alpha}_k)=\prox_{\hat{\alpha}_k
g}(\tilde{y}^k-\hat{\alpha}_k \nabla f(\tilde{y}^k))\in \dom g$. If
$\alpha_k<\al_{k-1}$ for $k>K$, it follows from  the definition of {\bf
Linesearch \ref{boundary}}  that
\begin{equation}\label{ac-eq-no-a}
\hat{\alpha}_k\left\|\nabla f\big(\hat{y}^k\big)-\nabla
f(\tilde{y}^k)\right\|>\delta\left\|\hat{y}^k-\tilde{y}^k\right\|.
\end{equation}
 Due to
the fact $\nabla f$ is Lipschitz continuous on $\dom g$ with constant $L$, we get from \eqref{ac-eq-no-a} that $\hat\alpha_k
L\|\tilde{y}^k-\hat{y}^{k}\|> \delta
\|\tilde{y}^k-\hat{y}^{k}\|. $ Thus $ \alpha_k\ge
\frac{\delta\theta}{L} $, which is a contradiction. Hence $\al_k\ge \al_{k-1}$, \emph{i.e.},  $\al_k= \al_{k-1}$ for all $k>K$. This tells us that $\al_K=\al>0$ whenever $\al<\frac{\delta\theta}{L}$. Thus we always have  $\al>0$ and complete the proof.
\end{proof}
\noindent Let us complete the section with a direct consequence of
the above proposition and Theorem~\ref{l-rate-2}.
\begin{corollary}\label{with-Lip-2}
Let $(x^k)_{k\in \NN}$ be the sequence generated by {\bf Method
\ref{A3}}. Suppose that $S_*\neq \emptyset$ and the gradient of $f$ is Lipschitz continuous on $\dom g$. Then we have
\begin{equation*}\label{rate1-2}
(f+g)(x^k)-\min_{x\in
\HH}(f+g)(x)=\mathcal{O}{((k+1)^{-2})}.
\end{equation*}
\end{corollary}
%%%%%%%%%%%%%%%%%%%%%%%%%%%%%%%% Algorithm 2 %%%%%%%%%%%%%%%%%%%%%%%%%%%%%%%%%%%%%%%
%%%%%%%%%%%%%%%%%%%%%%%%%%%%%%%%%%%%%%%%%%%%%%%%%%%%%%%%%%%%%%%%%%%%%%%%%%%%%%%%%%%%
\section{The forward-backward method with Linesearch \ref{boundary2}}\label{Sec-new-4}
{\bf Method~\ref{A1}} requires to evaluate the resolvent of
$\partial g$ inside {\bf Linesearch~\ref{boundary}} at each step of
the iteration. When the proximal step is not easy to  compute,  {\bf
Method~\ref{A1}} may be inefficient. To overcome this drawback, we
propose here a modification of the forward-backward method
by using {\bf Linesearch~\ref{boundary2}}, which involves only one
computation of the resolvent of $\partial g$ for all steps of this
linesearch. We also prove that the sequence generated by this method
is weakly convergent to a solution of problem \eqref{prob}.
\begin{center}\fbox{\begin{minipage}[b]{\textwidth}
\begin{method}\label{A2}
\item [    ] {\bf Initialization Step.} Take $x^0\in \dom g$ and $\theta\in(0,1)$.

\item [    ] \noindent {\bf Iterative Step.} Set
\begin{eqnarray}
J_k&=&\prox_{ g}(x^k-\nabla f(x^k))\label{tag4}\\
x^{k+1}&=&x^k -\beta_k (x^k-J_k)\label{paso3}
\end{eqnarray}
with $\beta_k:=$ {\bf Linesearch~\ref{boundary2}}$(x^k,\theta)$.
\item [    ] \noindent  {\bf Stop Criteria.} If $x^{k+1}=x^k$, then stop.
\end{method}\end{minipage}}\end{center}
Thanks to  Lemma~\ref{boundary-well2} and the convexity of $g$, we
note that $x^{k}\in \dom g$ inductively. Moreover,  it follows  from
{\bf Linesearch~\ref{boundary2}} that
\begin{equation}\label{12*}
(f+g)(x^{k+1})\le (f+g)(x^k)-\beta_k\left[g(x^k)-g(J_k)\right]-\beta_k
\la\nabla f(x^k), x^k-J_k\ra+\frac{\beta_k}{2}\|x^k-J_k \|^2.
\end{equation}
Next we obtain some similar results for {\bf Method~\ref{A2}} to the
ones in Section~3 for {\bf Method~\ref{A1}}. The following
proposition is corresponding to Proposition~\ref{prop1}.
\begin{proposition} \label{lema-para-qF}
Let $x\in \dom g$. Then we have
$$\|x^{k+1}-x\|^2\le\|x^{k}-x\|^2
+2\left[(f+g)(x^k)-(f+g)(x^{k+1})\right]+2\beta_k\left[(f+g)(x)-(f+g)(x^k)\right],\quad
\forall\, k\in \NN.$$
\end{proposition}
\begin{proof}
Fix any $x\in \dom g$ and set $ A_k:=\|x^{k+1}-x^k\|^2+\|x^{k}-x\|^2
-\|x^{k+1}-x\|^2=2\la x^{k}-x^{k+1},x^k-x\ra. $ Moreover, we get
from \eqref{paso3}  that
\begin{align*} \dfrac{A_k}{2\beta_k}&=\la x^k-J_k,x^k-x\ra
=\la \nabla f(x^k),x^k-x\ra +\la x^k-J_k-\nabla f(x^k),x^k-x\ra \\
&=\la \nabla f(x^k),x^k-x\ra +\la x^k-J_k-\nabla f(x^k),J_k-x\ra+\la
x^k-J_k-\nabla f(x^k), x^k-J_k\ra\\ &=\la \nabla f(x^k),x^k-x\ra
+\la x^k-J_k-\nabla f(x^k),J_k-x\ra - \la \nabla f(x^k), x^k-J_k\ra+
\|x^k-J_k\|^2.
\end{align*}
Observe from \eqref{tag4} that  $x^k-\nabla f(x^k)-J_k\in \partial g(J_k)$. By applying \eqref{sub-inq} and  \eqref{12*} to the above expression, we have
\begin{align*}
\dfrac{A_k}{2\beta_k}&\ge f(x^k)-f(x)+ g(J_k)-g(x) - \la \nabla
f(x^k), x^k-J_k\ra+ \|x^k-J_k\|^2\\ &\ge f(x^k)+
g(J_k)-(f+g)(x)+\frac{1}{\beta_k}\Big[(f+g)(x^{k+1})-(f+g)(x^k)\Big]+g(x^k)-g(J_k)+\frac{1}{2}\|x^k-J_k\|^2\\
&=
\left[(f+g)(x^k)-(f+g)(x)\right]+\frac{1}{\beta_k}\left[(f+g)(x^{k+1})-(f+g)(x^k)\right]+\frac{1}{2}\|x^k-J_k\|^2.
\end{align*}
It follows  that
\begin{align*}\|x^{k+1}-x\|^2\le&\|x^{k}-x\|^2+\|x^{k+1}-x^k\|^2-\beta_k\|x^k-J_k\|^2
+2\left[(f+g)(x^k)-(f+g)(x^{k+1})\right]\\&+2\beta_k\left[(f+g)(x)-(f+g)(x^k)\right].
\end{align*}
Since $x^{k+1}-x^k=\beta_k(J_k -x^k)$ by \eqref{paso3} and $\beta_k^2\le \beta_k$, we
conclude that
\begin{align*}\nonumber
\|x^{k+1}-x\|^2\le&\|x^{k}-x\|^2+(\beta_k^2-\beta_k)\|x^k-J_k\|^2
+2\left[(f+g)(x^k)-(f+g)(x^{k+1})\right]\\&+2\beta_k\left[(f+g)(x)-(f+g)(x^k)\right]\nonumber\\
\le&\|x^{k}-x\|^2
+2\left[(f+g)(x^k)-(f+g)(x^{k+1})\right]+2\beta_k\left[(f+g)(x)-(f+g)(x^k)\right]\end{align*}
as desired. The proof is complete. \end{proof}
\noindent It is worth
noting that using Proposition~\ref{lema-para-qF} with $x=x^k\in \dom
g$ gives us that
\begin{equation}\label{fkdec}
(f+g)(x^k)-(f+g)(x^{k+1})\ge \frac{1}{2} \|x^{k+1}-x^k\|^2\ge 0,
\end{equation}
which shows that {\bf Method~\ref{A2}} is also a descent method. 

Next we
establish the main result of this section whose statement is similar
to Theorem~\ref{new-cov}.
%%%%%%%%%%%%%%%%%%%%%%%%%%%%%%%%%%%%%%%%%%%%%%%%%%%%%%%%%%%%%%%%%%%
\begin{theorem}\label{ptos-de-acum2}
Let $(x^k)_{k\in\NN}$ be the sequence generated by {\rm {\bf
Method~\ref{A2}}}. The  following statements hold:

\item {\bf (i)} If $S_*\neq\emptyset$ then $(x^k)_{k\in \NN}$ is quasi-Fej\'er convergent to $S_*$ and  weakly converges to a point in
$S_*$. 
%Moreover, \begin{equation}\label{min*2}\lim_{k\to\infty}
%	(f+g)(x^k)=\min_{x\in \HH}\,(f+g)(x).\end{equation}
\item {\bf (ii)} If $S_*=\emptyset$ then we have
\begin{equation}\label{inf*}
\lim_{k\to\infty}\|x^k\|=+\infty \quad \mbox{and}
\quad\lim_{k\to\infty} (f+g)(x^k)=\inf_{x\in \HH}(f+g)(x).
\end{equation}
\end{theorem}\begin{proof}
To justify {\bf (i)}, suppose that $S_*\neq\emptyset$. By employing
Proposition~\ref{lema-para-qF}  at $x=x_*\in S_*\subseteq \dom g$, we have
\begin{equation}\label{fjj}
\|x^{k+1}-x_*\|^2\le\|x^{k}-x_*\|^2
+2\left[(f+g)(x^k)-(f+g)(x^{k+1})\right]\quad \mbox{for all}\quad k\in \NN.
\end{equation}
It follows from \eqref{fkdec} that
$\epsilon_k:=2\left[(f+g)(x^k)-(f+g)(x^{k+1})\right]\ge 0$.
Moreover, observe that
\begin{align*}
\sum_{k=0}^\infty\epsilon_k=&
2\sum_{k=0}^\infty\Big[(f+g)(x^k)-(f+g)(x^{k+1})\Big] \le
2\Big[(f+g)(x^0)-\lim_{k\to\infty} (f+g)(x^{k+1})\Big]\\ \leq&
2\Big[(f+g)(x^0)-(f+g)(x_*)\Big]< +\infty.
\end{align*}
 This
together with \eqref{fjj} tells us that the sequence $(x^k)_{k\in
\NN}$ is quasi-Fej\'er convergent to $S_*$ via
Definition~\ref{def-fejer}. By Fact \ref{lema-Fejer}(i), this sequence is bounded and hence it has weak accumulation points.
Let $\bar{x}$ be a weak accumulation point of
$(x^k)_{k\in\NN}$. Hence there exists a subsequence $(x^{n_k})_{k\in
\NN}$ of $(x^k)_{k\in\NN}$ converging weakly to $\bar{x}$.  Now we
distinguish our analysis into two cases.

\medskip

\noindent {\bf Case 1.} The sequence $\left(\beta_{n_k}\right)_{k\in
\NN}$ does not converge to $0$, \emph{i.e.}, there exist some
$\beta>0$ and a subsequence of $\left(\beta_{n_k}\right)_{k\in \NN}$
(without relabelling) such that
\begin{equation}\label{bk-no-0}
\beta_{n_k}\geq\beta,\quad \forall\, k\in \NN.
\end{equation}By using Proposition~\ref{lema-para-qF} with $x=x_*\in S_*$, we get
\begin{align}\nonumber
\beta_k
\left[(f+g)(x^k)-(f+g)(x_*)\right]\le&\frac{1}{2}(\|x^{k}-x_*\|^2-\|x^{k+1}-x_*\|^2)
+(f+g)(x^k)-(f+g)(x^{k+1}).
\end{align}
 Summing from $k=0$ to $m$ in the above
inequality implies
\begin{align*}
\sum_{k=0}^m\beta_k
\left[(f+g)(x^k)-(f+g)(x_*)\right]\le&\frac{1}{2}(
\|x^0-x_*\|^2-\|x^{m+1}-x_*\|^2)+(f+g)(x^0)-(f+g)(x^{m+1})\\
\le&\frac{1}{2}\|x^0-x_*\|^2+(f+g)(x^0)-(f+g)(x_*).
\end{align*}
By taking $m\to \infty$ and using the fact that $(f+g)(x^k)\ge (f+g)(x_*)$, we obtain that
$$
\sum_{k=0}^{\infty}\beta_{n_k}\left[(f+g)(x^{n_k})-(f+g)(x_*)\right]\le \sum_{k=0}^{\infty}\beta_k\left[(f+g)(x^k)-(f+g)(x_*)\right]<+\infty,
$$
which together with \eqref{bk-no-0} establishes  that
$
\dsty\lim_{k\rightarrow\infty}\,(f+g)(x^{n_k})=(f+g)(x_*).
$ Since $f+g$ is  lower semicontinuous on $\dom g$, it is also weakly \emph{l.s.c.} due to the convexity of $f+g$. It follows from
the last equality that
\[ (f+g)(x_*)\le (f+g)(\bar{x})\le
\liminf_{k\rightarrow\infty}(f+g)(x^{n_k})=\lim_{k\rightarrow\infty}
(f+g)(x^{n_k})=(f+g)(x_*),
\]
 which yields $(f+g)(\bar{x})=(f+g)(x_*)$
and thus $\bar{x}\in S_*$.

\medskip

\noindent{\bf Case 2.} $\dsty\lim_{k\rightarrow\infty}\beta_{k}=0$.
Define $\dsty \hat{\beta}_k:=\frac{\beta_k}{\theta}>0$ and
\begin{equation}\label{2*paso}
\hat{y}^k:=x^k-\hat{\beta}_k
(x^k-J_k)=(1-\hat{\beta}_k)x^k+\hat{\beta}_kJ_k.
\end{equation}
It follows from  the definition of {\bf Linesearch 2}
 that
\begin{equation}\label{no-armijo}
(f+g)(\hat{y}^k)>(f+g)(x^k)-\hat{\beta}_k[g(x^k)-g(J_k)]-\hat{\beta}_k\la\nabla
f(x^k), x^k-J_k\ra +\frac{\hat{\beta}_k}{2}\|x^k-J_k\|^2.
\end{equation}
This together with \eqref{sub-inq} and \eqref{2*paso} gives us that
\begin{align*}
0>&\,-\hat{\beta}_k\la\nabla f(x^k),
x^k-J_k\ra+(f+g)(x^k)-(f+g)(\hat{y}^k)-\hat{\beta}_k[g(x^k)-g(J_k)]+\frac{\hat{\beta}_k}{2}\|x^k-J_k\|^2\\
=&\,-\hat{\beta}_k\la\nabla f(x^k),
x^k-J_k\ra+f(x^k)-f(\hat{y}^k)+g(x^k)-g(\hat{y}^k)-\hat{\beta}_k[g(x^k)-g(J_k)]+\frac{\hat{\beta}_k}{2}\|x^k-J_k\|^2\\
\geq&-\hat{\beta}_k\la\nabla
f(x^k), x^k-J_k\ra+\la\nabla f(\hat{y}^k),
x^k-\hat{y}^k\ra+\frac{\hat{\beta}_k}{2}\|x^k-J_k\|^2\\&+g(x^k)-(1-\hat{\beta}_k)g(x^k)-\hat{\beta}_kg(J_k)-\hat{\beta}_k[g(x^k)-g(J_k)]\\
=&~\hat{\beta}_k\dsty\la \nabla f(\dsty \hat{y}^k)-\nabla
f(x^k),x^{k}-J_k\ra+\frac{\hat{\beta}_k}{2}\|x^k-J_k\|^2.
\end{align*}
We obtain  that
$$
\disp\frac{\hat\beta_k}{2}\|x^k-J_k\|^2 <\dsty~ \hat\beta_k\| \nabla
f(\hat{y}^k)-\nabla f(x^k)\| \cdot \| x^k-J_k\|,
$$
which yields
\begin{equation}\label{bb}
\frac{1}{2}\| x^k-J_k\|\le \| \nabla f(\hat{y}^k)-\nabla f(x^k)\|.
\end{equation}
Since $\prox_g(\cdot)$ is nonexpansive, we
get from \eqref{tag4} that
$
\|J_k-J_0\|\le \|x^k-x^0\|+\|\nabla f(x^k)-\nabla f(x^0)\|.
$
Due to Assumption {\bf A2} and the boundedness of $(x^k)_{k\in
\NN}$, the latter tells us that $(J_k)_{k\in \NN}$ is also bounded. This together with
\eqref{2*paso} and the fact $\beta_k\to 0$ implies  that $ \|\hat y^k -x^k\|\to0$ as $k\to \infty$. Since $\nabla
f$ is uniformly continuous on bounded sets, we get $\| \nabla
f(\hat{y}^k)-\nabla f(x^k)\| \to0$ as $k\to\infty$ and derive from
\eqref{bb} that
\begin{equation}\label{limite1_DI}
\dsty\lim_{k\rightarrow\infty}\,\| x^{k}-J_k\| = 0,
\end{equation}
Since $\nabla f$ is
uniformly continuous on bounded sets, \eqref{limite1_DI} implies
\begin{equation}\label{grad-to-0}
\dsty\lim_{k\rightarrow\infty}\| \nabla f(x^{k})-\nabla f(J_k)\| =0.
\end{equation}
Using \eqref{in-sub} with $z=x^k-\nabla f(x^k)$ gives us that
\begin{equation*}\label{inTk}
x^k-J_k + \nabla f(J_k)-\nabla f(x^k)\in \nabla f(J_k) +
\partial g(J_k)\subseteq\partial (f+g)(J_k).
\end{equation*}
By passing to the limit over the subsequence $(n_k)_{k\in \NN}$ in
the above inclusion, we get from  Fact~\ref{teo-p1}, \eqref{limite1_DI}, and \eqref{grad-to-0} that $0\in \partial (f+g)(\bar{x})$,
which implies $\bar{x}\in S_*$.

In all possible cases above,  any weak accumulation point of
$(x^k)_{k\in \NN}$ belongs to $S_*$.
Fact~\ref{lema-Fejer}(ii) tells us that $(x^k)_{k\in \NN}$
converges weakly to an optimal solution in $S_*$.
Thus this completes the proof of {\bf (i)}. Moreover, the proof of part {\bf (ii)} is quite
similar to the arguments used to prove Theorem~\ref{new-cov}(ii). We
omit the detail and complete the proof.
\end{proof}
\noindent From the view of \eqref{inf*} and also our Theorem~\ref{new-cov}, it is natural to
question that whether
\begin{equation}\label{min}
\lim_{k\to\infty} (f+g)(x^k)=\min_{x\in \HH}\,(f+g)(x)
\end{equation}
in the case $S_*\neq \emptyset$. We do not know the answer in
general, but when either $f+g$ is continuous on the $\dom g$   in finite dimensions or the sequence $(\beta_k)_{k\in \NN}$ is bounded
below by a positive constant, the equality \eqref{min} is true with
some further complexity discussed in the next subsection.
%%%%%%%%%%%%%%%%%%%%%%%%%%%%%%%%%%%%%%%%%%%%%%%%%%%%%%%%%%%%%%%%%%%
%%%%%%%%%%%%%%%%%%%%%%%%%%%%%%%%%%%%%%%%%%%%%%%%%%%%%%%%%%%%%%%%%%%

\subsection{Complexity analysis of Method \ref{A2}}
In this subsection we establish the complexity of   {\bf
Method~\ref{A2}}  with a similar rate to Theorem~\ref{l-rate} as follows.
\begin{theorem}\label{l-rate-3}
Let $(x^k)_{k\in \NN}$ and $\left(\beta_k\right)_{k\in \NN}$ be the
sequences generated in {\bf Method~\ref{A2}}. Suppose that $S_*\neq \emptyset$ and  there is
some $\beta>0$ satisfying $\beta_k\ge\beta>0$ for all $k\in \NN$.
Then for all $k\in \NN$
\begin{equation}\label{rate-3}
(f+g)(x^k)-\min_{x\in \HH}\,(f+g)(x)\le\frac{1}{2\beta } \frac{[{\rm
dist}(x^0,S_*)]^2+2\left[(f+g)(x^0)-\dsty\min_{x\in
\HH}(f+g)(x)\right]}{k}.
\end{equation}
If in addition $\dim \HH<+\infty$ then we have 
\begin{equation}\label{rate4}
\lim_{k\to\infty}k\left[(f+g)(x^k)-\min_{x\in \HH}\,(f+g)(x)\right]=0. 
\end{equation}
\end{theorem}
\begin{proof}
By using Proposition~\ref{lema-para-qF}, at $\ell\in \NN$ and $x_*\in
S_*$, we get
\begin{align}\nonumber
0\ge&(f+g)(x_*)-(f+g)(x^{\ell+1})\\\nonumber\ge&\frac{1}{2\beta_\ell}\left(
\|x^{\ell+1}-x_*\|^2-\|x^\ell-x_*\|^2+2\left[(f+g)(x^{\ell+1})-(f+g)(x^{\ell})\right]\right)\\
\ge & \frac{1}{2\beta}\left(
\|x^{\ell+1}-x_*\|^2-\|x^\ell-x_*\|^2+2\left[(f+g)(x^{\ell+1})-(f+g)(x^{\ell})\right]\right)\label{antes-sum-3}
\end{align}
for all $\ell\in \NN$. Summing the above inequality \eqref{antes-sum-3}, over $\ell=0,1,\ldots,k-1$, we have
\begin{align}\label{p-sum2-3}\nonumber
&\sum_{\ell=0}^{k-1}\left[(f+g)(x_*)-(f+g)(x^{\ell+1})\right]\ge\frac{1}{2\beta}\left(
\|x^{k}-x_*\|^2-\|x^0-x_*\|^2+2[(f+g)(x^{k})-(f+g)(x^{0})]\right)\\&\ge
\frac{1}{2\beta}\left(
\|x^{k}-x_*\|^2-\|x^0-x_*\|^2+2\left[(f+g)(x_*)-(f+g)(x^{0})\right]\right).
\end{align}
Noting that $(f+g)(x^{\ell+1})\ge (f+g)(x^{\ell})$ for all $\ell=0,\ldots, k-1$ by \eqref{fkdec}, we obtain from \eqref{p-sum2-3} that 
\begin{equation*}
k\left[(f+g)(x_*)-(f+g)(x^k)\right]\ge
\frac{1}{2\beta}\left(
\|x^{k}-x_*\|^2-\|x^0-x_*\|^2+2\left[(f+g)(x_*)-(f+g)(x^{0})\right]\right),
\end{equation*}
which clearly implies the following expression
\begin{equation}\label{ratex}
(f+g)(x^k)-(f+g)(x_*)\le\frac{1}{2\beta }
\frac{\|x^0-x_*\|^2+2\left[(f+g)(x^0)-(f+g)(x_*)\right]}{k}
\end{equation}
for all $x^*\in S_*$.  \eqref{rate-3} is obtained.

To justify \eqref{rate4} when $\dim \HH<+\infty$, suppose that $(x^k)_{k\in\NN}$ converges (strongly) to some $x_*\in S_*$ by Theorem~\ref{ptos-de-acum2}. Hence for any $\ve>0$ there exists some $K>0$ such that  
\begin{equation}\label{xx}
\|x^k-x_*\|\le \ve \quad \mbox{and}\quad (f+g)(x^k)-(f+g)(x_*)\le \ve\quad \mbox{for all}\quad k\ge K, 
\end{equation}
where the second inequality follows from  the recent estimate \eqref{ratex}. Adding \eqref{antes-sum-3} for $\ell=K, K+1,\ldots, K+k-1$ and noting that 
\begin{align*}
\sum_{\ell=K}^{K+k-1}\left[(f+g)(x_*)-(f+g)(x^{\ell+1})\right]\ge \frac{1}{2\beta}\left(\|x^{K+k}-x_*\|^2-\|x^K-x^*\|^2+2\left[(f+g)(x^{K+k})-(f+g)(x^{K})\right]\right).
\end{align*}
Since $(f+g)(x^{\ell+1})\ge (f+g)(x^{K+k})$ for all $\ell=K, K+1,\ldots, K+k-1$ by \eqref{fkdec}, we get from the latter and \eqref{xx} that 
\begin{equation*}
k\big[(f+g)(x_*)-(f+g)(x^{K+k})\big]\disp\ge \frac{1}{2\beta}\Big(-\|x^K-x^*\|^2+2\big[(f+g)(x_*)-(f+g)(x^K)\big]\Big)\disp\ge \frac{1}{2\beta}(-\ve^2-2\ve). 
\end{equation*}
It follows that 
\begin{align*}
\limsup_{k\to\infty}\, k\big[(f+g)(x^{k})-(f+g)(x_*)\big]&=\limsup_{k\to\infty}\, (K+k)\big[(f+g)(x^{K+k})-(f+g)(x_*)\big]\\
&\le\limsup_{k\to\infty}\,\frac{K+k}{k}\cdot\frac{\ve^2+2\ve}{2\beta}=\frac{\ve^2+2\ve}{2\beta}. 
\end{align*}
Since this inequality holds for any $\ve>0$, we have 
$
\limsup_{k\to\infty}\, k\big[(f+g)(x^{k})-(f+g)(x_*)\big]\le 0. 
$
Note that $(f+g)(x^{k})-(f+g)(x_*)\ge 0$ for all $k\in \NN$, we get  \eqref{rate4} and thus complete the proof of theorem.
\end{proof}

%%%%%%%%%%%%%%%%%%%%%%%%%%%%%%%%%%%%%%%%%%%%%%%%%%%%%%%%%%%%%%%%%%%%%%

%%%%%%%%%%%%%%%%%%%%%%%%%%%%%%%%%%%%%%%%%%%%%%%%%%%%%%%%%%%%%%%%%%%%%%
Similarly to Lemma~\ref{lema-alpha}, we present some  sufficient conditions for the below boundedness by a positive constant of the stepsize
generated by {\bf Linesearch \ref{boundary2}}.
\begin{proposition}\label{lema-alpha2}
Let $(\beta_k)_{k\in \NN}$ be the sequence generated by {\bf
Linesearch~\ref{boundary2}} on {\bf Method~\ref{A2}}. The following
statements hold:

\item {\bf (i)} If the gradient
of $f$ is globally Lipschitz continuous on $\dom g$ with constant $L>0$, then
$\beta_k\ge \min\left\{1, \frac{\theta}{2L}\right\}$ for
all $k\in \NN$.

\item {\bf (ii)} Suppose that $\dim \HH<+\infty$ and $S_*\neq\emptyset$. If $\nabla f$ is locally Lipschitz continuous at any
$x\in S_*$ then there exists  $x_*\in S_*$ such that
\begin{equation}\label{eng}
\liminf_{k\to\infty}\beta_k\ge
\min\Big\{1,\frac{\theta}{2\mathcal{L}}\Big\},
\end{equation}
where $\mathcal{L}>0$ is a Lipschitz constant of $\nabla f$ around
$x_*$. Consequently, there exists $\beta>0$ such that $\beta_k\ge
\beta$ for all $k\in \NN$.
\end{proposition}
\begin{proof}
First let us verify {\bf(i)} by supposing that  the gradient
of $f$ is globally Lipschitz continuous on $\dom g$ with constant $L>0$. Define $\hat{\beta}_k:=\dsty
\frac{\beta_k}{\theta}>0$ and
\begin{equation}\label{hatxk**-3}
\hat{y}^k:=\hat{\beta}_kJ_k+(1-\hat{\beta}_k)x^k=x^k-\hat{\beta}_k(x^k-J_k).
\end{equation}
If $\beta_k<1$, we get from
 {\bf Linesearch~2} that
$$
(f+g)(\hat{y}^k)>(f+g)(x^k)-\hat{\beta}_k[g(x^k)-g(J_k)]-\hat{\beta}_k\la\nabla
f(x^k), x^k-J_k\ra +\frac{\hat{\beta}_k}{2}\|x^k-J_k\|^2,
$$
which together with \eqref{hatxk**-3} and  that $x^k-J_k\neq
0$ implies that $\hat y^k\neq x^k$. Furthermore, it is similar to
\eqref{bb} in the proof of Theorem~\ref{ptos-de-acum2} that  $
\frac{1}{2}\| x^k-J_k\|\le \| \nabla f(\hat{y}^k)-\nabla f(x^k)\|$.
Due to the Lipschitz continuity with constant $L$ of $\nabla f$, we
get from the latter and \eqref{hatxk**-3} that
\[
 \frac{1}{2}\|
x^k-J_k\|\le L \|x^{k}-\hat{y}^{k}\|=L\hat \beta_k \| x^k-J_k\|.
\]
Since $x^k-J_k\neq 0$, the inequality above yields $ \hat \beta_k
\ge \frac{1}{2L}$ and thus $\beta_k\ge \frac{\theta}{2L}$ when $\beta_k<1$. It follows that $\beta_k\ge \min\{1,\frac{\theta}{2L}\}$ as desired.

To verify the second part, suppose that $\dim \HH<+\infty$,
$S_*\neq\emptyset$, and that  $\nabla f$ is locally Lipschitz
continuous at any $x\in S_*$. By Theorem~\ref{ptos-de-acum2}, suppose that $(x^k)_{k\in \NN}$ (strongly) converges to $x_*\in S_*$. Hence there exist $\ve, \mathcal{L}>0$ such that 
\[
\|\nabla f(x)-\nabla f(y)\|\le \mathcal{L}\|x-y\|\quad \mbox{for all}\quad x,y\in \mathbb{B}_\ve(x_*).
\]
Since $x^k\to x_*$ as $k\to \infty$, we find $K>0$ such that $\|x^k-x_*\|\le \ve$ for all $k>K$. Pick any $k>K$, if $\beta_k<1$, define $\hat{\beta}_k=\dsty
\frac{\beta_k}{\theta}>0$ and
$\hat{y}^k=\hat{\beta}_kJ_k+(1-\hat{\beta}_k)x^k$. Similarly to the above argument of
the first part, we have $x^k-J_k\neq 0$ and
\begin{equation}\label{nm}
\frac{1}{2}\| x^k-J_k\|\le \| \nabla f(\hat{y}^k)-\nabla f(x^k)\|.
\end{equation}
We consider two cases as in Theorem~\ref{ptos-de-acum2} as below:

\noindent {\bf Case 1.}  The sequence $(\beta_k)_{k\in \NN}$ is bounded below by a positive number $\beta>0$. Thanks to \eqref{paso3} we have 
\[
 \|x^k-J_k\|= \frac{\|x^k-x^{k+1}\|}{\beta_k}\le\frac{\|x^k-x^{k+1}\|}{\beta}\to 0
\]
as $k\to \infty$. It follows that 
$\disp
\|x^k-\hat y^k\|=\frac{\beta_k}{\theta}\|x^k-J_k\|\to 0,
$
which tells us that $(\hat y^k)_{k\in\NN}$  is  converging to
$x_*$. Hence there exists $K_1>K$ such
that $\hat y^k\in \mathbb{B}_\ve(x_*)$  for all $k>K_1$. By combining this with 
\eqref{nm},  we derive 
\[
\frac{1}{2}\| x^k-J_k\|\le \mathcal{L}
\|\hat{y}^k-x^k\|=\mathcal{L}\hat \beta_k \|x^k-J_k\|\quad \mbox{for
all}\quad k>K_1.
\]
Since $\| x^k-J_k\|\neq 0$, the latter gives us that  $\frac{1}{2}\le \mathcal{L}\hat
\beta_k$, i.e., $\beta_k\ge \frac{\theta}{2\mathcal{L}}$ for
all $k>K_1$.

\noindent {\bf Case 2.} The sequence $(\beta_k)_{k\in \NN}$ is not bounded below by a positive number $\beta$. Hence we may find a subsequence (no labeling)  $(\beta_k)_{k\in \NN}$ converging to $0$. It is similar to the proof of {\bf Case 2} in Theorem~\ref{ptos-de-acum2} that  $(x^k)_{k\in\NN}$ and $(J_k)_{k\in\NN}$ are bounded. It follows that 
\[
\lim_{k\to\infty} \|x^k-\hat y^k\|=\lim_{k\to\infty}\frac{\beta_k}{\theta}\|x^k-J_k\|= 0.
\]
Thus  the sequence $(\hat y^{k})_{k\in\NN}$ is converging to
$x_*$. Repeating the corresponding part in the proof of {\bf Case 1} above, we also have  $\beta_k\ge \frac{\theta}{2\mathcal{L}}$ for any large $k$, which is the contradiction.

From the analysis of both cases above, we find $K_1>0$ such that $\beta_k\ge \frac{\theta}{2\mathcal{L}}$ if $\beta_k<1$ for any $k>K_1$. This means  $\beta_k\ge \min\{1,\frac{\theta}{2\mathcal{L}}\}$ for $k\ge K_1$. The proof is complete. 
\end{proof}

Let us complete the section by presenting a corresponding corollary
to Corollary~\ref{with-Lip}, which is easily derived from
Theorem~\ref{l-rate-3} and Proposition~\ref{lema-alpha2}.
\begin{corollary}
Let $(x^k)_{k\in \NN}$ be the sequence generated by
{\bf Method~\ref{A2}}. Suppose that $S_*\neq \emptyset$. 

\item {\bf (i)} If the gradient of $f$ is globally Lipschitz continuous  on $\dom g$, then 

$$(f+g)(x^k)-\min_{x\in
	\HH}(f+g)(x)=\mathcal{O}(k^{-1}).$$

\item {\bf(ii)} If $\dim \HH<+\infty$ and the gradient of $f$ is locally  Lipschitz continuous  on
$S_*$, then we have
$$(f+g)(x^k)-\min_{x\in
\HH}(f+g)(x)=o(k^{-1}).$$
\end{corollary}

%%%%%%%%%%%%%%%%%%%%%%%%%%%%%%%%%%%%%%%%%%%%%%%%%%%%%%%%%%%%%%%%%%%%%%%%%%%%%%%%%%%%
%%%%%%%%%%%%%%%%%%%%%%%%%%%%%%%% Algorithm 3 %%%%%%%%%%%%%%%%%%%%%%%%%%%%%%%%%%%%%%%
%%%%%%%%%%%%%%%%%%%%%%%%%%%%%%%%%%%%%%%%%%%%%%%%%%%%%%%%%%%%%%%%%%%%%%%%%%%%%%%%%%%%

\section{Conclusions}\label{s:6}
In Hilbert spaces, it is well-known that convexity on both functions
and global Lipschitz continuity on the gradient of $f$ are
sufficient for providing convergence of the sequence generated by
the forward-backward splitting methods for solving problem
\eqref{prob}. However, the Lipschitz assumption is usually a
restriction in many particular circumstances. In this work we dealt
with weak convergence of the forward-backward splitting
method for convex optimization problems by taking the advantage of
the  linesearches. This  not only eliminates the serious drawback of
estimating the Lipschitz constant  to choose the stepsize in
\eqref{F-B-I} but also establishes many complexity results without
imposing the Lipschitz assumption. Our schemes through the
linesearches provide   rigorous and implementable ways of updating
the iterates, which can be easily adapted for applications.

We hope that this study will serve as a basis for future research on
other efficient variants of the forward-backward splitting
iteration. In particular we find possibility to develop our methods
to the descent coordinate gradient method
\cite{nest-SIAM} for solving structured
convex optimization problems. Moreover, we discuss in separate
papers the cases when $f$ or $g$ are nonconvex following the ideas
exposed in \cite{bot-2014} and even removing the differentiability
of $f$ and adding dynamic choices of the stepsizes with conditional
and deflected techniques combining the ideas in
\cite{yun-2014,cond,francolli}. We are also looking to the
incremental (sub)gradient method like \cite{nedic-SIAM-2001} for
problem \eqref{prob}, when $f$ is the sum of a large number of
functions. {An interesting project, suggested by a referee, that we are pursuing is to study possible complexity $o(k^{-2})$ and the weak convergence  of {\bf Method2} without assuming the global Lipschitz continuity on the gradient of the smooth function as in \cite{attouch,chambolle}.} \vspace*{-0.1in}

\subsection*{ACKNOWLEDGMENTS} This work was partially completed while the authors were visiting University of British Columbia Okanagan (UBCO). The authors are
grateful to the Irving K. Barber School of Arts and Sciences at UBCO
and particularly to  Heinz H. Bauschke and Shawn Wang for the generous hospitality.  {We also would like to express our gratitude to two anonymous referees for many useful suggestions, which allowed us to significantly improve the original presentation.}

\end{document}